\documentclass[a4paper]{amsart}
\newif\ifcomments

\usepackage{amsmath, amsthm, amsfonts,stmaryrd,amssymb }

\usepackage{mathrsfs}                      
\usepackage{mathtools}                     
\usepackage{booktabs}
\usepackage{todonotes}
\usepackage[all]{xy}
\usepackage[utf8]{inputenc}
\usepackage{mathtools}                     
\usepackage{todonotes}
\usepackage{faktor}
\usepackage{ marvosym }
\usepackage{bbm}                         
\usepackage{framed}
\usepackage{graphicx}   
\usepackage[top=1in, bottom=1.25in, left=1.25in, right=1.25in]{geometry}
\usepackage{theoremref}

\numberwithin{equation}{section}
\theoremstyle{plain}
\newtheorem{theorem}{Theorem}[section]

\newtheorem{remark}[theorem]{Remark}
\newtheorem{lemma}[theorem]{Lemma}
\newtheorem{corollary}[theorem]{Corollary}
\newtheorem{proposition}[theorem]{Proposition}

\newtheorem{notation}[theorem]{Notation}
\theoremstyle{definition}
\newtheorem{definition}[theorem]{Definition}

\def\ph{\varphi}

\def\o{\circ} 
\def\i{^{-1}} 
\def\x{\times}

\def\R{{\mathbb R}}
\def\N{{\mathbb N}}

\def\ad{\operatorname{ad}} 
\def\Ad{\operatorname{Ad}}
\def\exp{\operatorname{exp}}

\let\on=\operatorname

\newcommand{\ud}{\,\mathrm{d}}

\let\mc=\mathcal

\newcommand{\eqdef}{\ensuremath{\stackrel{\mbox{\upshape\tiny def.}}{=}}}
\begin{document}
\date{\today}
\title[Riemannian cubics on the group of diffeomorphisms]{Riemannian cubics on the group of diffeomorphisms and the Fisher-Rao metric}
\author{Rabah Tahraoui} 
\address{UMR 6085 CNRS-Université de Rouen}
\email{tahraoui@ceremade.dauphine.fr}
\author{Fran\c{c}ois-Xavier Vialard}
\address{Universit\'e Paris-Dauphine, PSL Research University, Ceremade \\ INRIA, Project team Mokaplan}
\email{fxvialard@normalesup.org}

\begin{abstract}
We study a second-order variational problem on the group of diffeomorphisms of the interval $[0,1]$ endowed with a right-invariant Sobolev metric of order $2$, which consists in the minimization of the acceleration. We compute the relaxation of the problem which involves the so-called Fisher-Rao functional, a convex functional on the space of measures. This relaxation enables the derivation of several optimality conditions and, in particular, a sufficient condition which guarantees that a given path of the initial problem is also a minimizer of the relaxed one. This sufficient condition is related to the existence of a solution to a Riccati equation involving the path acceleration.
\end{abstract}

\maketitle

\section{Introduction}
This paper is concerned with a variational problem on the group of diffeomorphisms of the segment $[0,1]$, which consists in finding a curve on the group of minimal acceleration with prescribed or relaxed boundary conditions. The motivation for studying this problem is to give a theoretical ground to formal calculations made in \cite{HOSplines1,HOSplines2} and the numerical implementation presented in \cite{SinghVN15} together with applications to medical imaging.

\subsection{Motivation and previous works}

Riemannian cubics (also called Riemannian splines) and probably more famous, its constrained alternative called Elastica belong to a class of problems that have been studied since the work of Euler (see the discussion in \cite{MumfordElastica}). Let us present the variational problem in a Riemannian setting. Riemannian splines are minimizers of 
\begin{equation}\label{SplinesFunctional}
\mathcal{J}(x) = \int_0^1 g\left( \frac{D}{Dt} \dot{x},\frac{D}{Dt}\dot{x}  \right) \ud t\,,
\end{equation}
where $(M,g)$ is a Riemannian manifold, $\frac{D}{Dt}$ is its associated covariant derivative and $x $ is a sufficiently smooth curve from $[0,1]$ in $M$ satisfying first order boundary conditions, i.e. $x(0),\dot{x}(0)$ and $x(1),\dot{x}(1)$ are fixed.
The case of Elastica consists in restricting the previous optimization problem to the set of curves that are parametrized by unit speed (when the problem is feasible), namely $g(\dot{x},\dot{x})=1$ for all time. 

This type of variational problems has been several times introduced and studied in applied mathematics \cite{birkhoff65,VariationalStudySplines,Noakes1,splinesCk,Crouch,splinesanalyse,Koiso} as well as in pure mathematics \cite{Bismut,langer1984,Bryant} and it was then extensively used and numerically developed in image processing and computer vision \cite{MumfordElastica,Masnou1,PamiSplines,Samir,CAD,Chan02eulerselastica}. 
In the past few years, higher-order models have been introduced in biomedical imaging for interpolation of a time sequence of shapes. They have been introduced in \cite{TrVi2010} for a diffeomorphic group action on a finite dimensional manifold and further developed for general invariant higher-order lagrangians in \cite{HOSplines1,HOSplines2} on a group. A numerical implementation together with a generalized model has been proposed in \cite{SinghVN15} in the context of medical imaging applications. What is still unsolved in the case of a group of diffeomorphisms, is the question of existence and regularity where the main obstacle is caused by the infinite dimensional setting. 

\par In infinite dimensions, to the best of our knowledge, only the linear case has been addressed \cite{VarSplines}. Actually, in the non-linear case, namely the case of Riemannian metrics in infinite dimensions, existence of minimizing geodesics is already non-trivial as shown by Atkin in \cite{Atkin1997}, where an example of a geodesically complete Riemannian manifold is given such that the exponential map is not surjective. Therefore, Elastica or Riemannian splines will preferably be studied on Riemannian manifolds where all the properties of the Hopf-Rinow theorem fully holds.

Only recently it has been proven in \cite{BruverisVialard} that the group of diffeomorphisms endowed with a right-invariant Sobolev metric of high enough order is complete in the sense of the Hopf-Rinow theorem. Namely, the group is a complete metric space which implies that it is geodesically complete, i.e. geodesics can be extended for all time. Moreover, between any two elements in the group in the connected component of the identity, there exists a length minimizing geodesic. Motivated by this positive result, we explore in this paper the minimization of the acceleration in the one dimensional case, which is the first step towards its generalizations in higher dimensions. 

\subsection{Contributions}
In section \ref{Sec:Reduction}, we summarize the formal derivation of \cite{HOSplines1,HOSplines2} and show important issues making this formulation of difficult use in addressing the problem of existence of Riemannian splines. More precisely, curves of zero acceleration on the group of diffeomorphisms endowed with a right-invariant metric are geodesics. These geodesic equations are easily written in Eulerian coordinates where, as in the incompressible Euler equation, there is a loss of smoothness. Therefore in Section \ref{Sec:GeodesicEquation}, we take advantage of the point of view introduced by Ebin and Marsden in \cite{Ebin1970} in which the authors showed the smoothness of the metric in Lagrangian coordinates. We then detail the Hamiltonian  formulation of the geodesic equations on $\on{Diff}_0^2([0,1])$ with a right-invariant Sobolev metric of order $2$ written in Lagrangian coordinates. Passing by, we give a simple proof of the regularity of geodesics in terms of the smoothness of the endpoints.
\par
In Section \ref{Sec:Relaxation}, we compute the relaxation of the acceleration functional in Theorem \ref{Th:Relaxation}, where the Fisher-Rao convex functional introduced in Definition \ref{def:FR}, appears to play a key role. As is usual in relaxation, weak convergence leads to defect measures and in our situation, these defect measures appear in the Fisher-Rao functional. Existence of minimizers is guaranteed on the product space of paths and defect measures. The final formulation of the relaxation appears in formula \eqref{eq:FirstVariationalReduction}.
\par
In Section \ref{Sec:MainTheorem}, we derive standard optimality conditions by means of convex analysis in Proposition \ref{Th:FirstOrderCondition} as well as an explicit sufficient condition for optimality in Proposition \ref{Th:FirstOrderConditionSufficient}. We also obtain a \emph{weak} strict convexity result in Proposition \ref{Th:WeakConvexity} for the minimization in the defect measure variable. As a consequence of this sufficient optimality condition, we prove that there exist solutions of the initial problem. We are also able to construct examples of paths of the initial Hamiltonian system that are critical points of the initial acceleration functional but not minimizers with respect to the relaxed acceleration functional. These examples were achieved using numerical computations and were motivated by an explicit construction (see Section \ref{Sec:Numerics} and the preceding proposition).

\subsection{Notations}

\begin{itemize}
\item For a manifold $M$, we denote by $TM$ its tangent bundle.
\item If $f: \R \to \R$ is a differentiable function, $f'$ denotes its derivative.
\item If $f: M \mapsto N$ is a $C^1$ map between manifolds, we denote by $df:$ it differential.
\item If $f(t,x)$ is a real valued function from time and space, $\partial_x f$ denotes its derivative with respect to $x$ and $\partial_t f$ denotes its derivative with respect to $t$. Often, we will denote $\dot{f}$ the derivative of $f$ with respect to the time variable; it will be used to stress the fact that we will often work in time dependent quantities with values in a Hilbert manifold of functions.
\item The square $D=[0,1]^2$ represents time and space variables $(t,x)$.
\item $(C^0(D),\| \cdot \|_\infty)$ is the space of continuous functions on $D$ endowed with the sup norm.
\item The space of positive and finite Radon measures on $D$ will be denoted by $\mathcal{M}(D)$.
\item The subspace of measures $\mathcal{M}_0(D) \eqdef \{ \mu \, ; \, \mu_{t=0} = 0  \}$ where $\mu_{t=0}$ denotes the disintegration of $\mu$ at time $0$.
\item The space of smooth test functions on $D$ is $\mathfrak{D}$.
\item The space of continuous real functions which are $C^1$ with respect to the first (time) variable is denoted $C_{1,0}(D)$.
\item The bracket between test functions and distribution will be denoted by $\langle \cdot, \cdot \rangle$.
\item We will use the weak topology on $\mathcal{M}$ and it is metrizable and separable since $D$ is separable.
\item The symbol $\star$ denotes the convolution.
\item The space of bounded linear operators between normed vector spaces $E,F$ is denoted by $\mathcal{L}(E,F)$.
\item The identity diffeomorphism will be denoted by $\on{Id}$.
\end{itemize}

\section{First and higher-order Euler-Poincar\'e reduction}\label{Sec:Reduction}

\subsection{Euler-Lagrange equation for reduced lagrangians}
A prototypical example of the situation we are interested in is the case of the incompressible Euler equation.
As shown by Arnold in \cite{Arnold1966}, the incompressible Euler equation is the Euler-Lagrange equation of geodesics on the group of volume preserving diffeomorphisms for the $L^2$ right-invariant metric. By analogy with the Lie group point of view, the incompressible Euler equation in terms of the vector field is written on the tangent space at identity, $T_{\on{Id}}G$ for a Lie group $G$. The Lagrangian introduced in \cite{Arnold1966} was already written on the space of vector fields, however, it is  a particular case of Lagrangians that can be written by a change of variable only at the tangent space of identity $\mathfrak{g}\eqdef T_{\on{Id}}G$, the Lie algebra in finite dimension\footnote{Note that in general in infinite dimensions, $G$ is generally not a Lie group, see \cite{Omori1978} for more details.}. This class of Lagrangians leads to the so-called Euler-Poincaré or Euler-Arnold equation when the Euler-Lagrange equation is written on $T_{\on{Id}}G$. A short proof of the derivation of this equation is given in \cite[Theorem 3.2]{MisolekPreston2010} in the case of a kinetic energy but let us underline that the same equation holds true for general Lagrangian that are right-invariant. We will need the definition of the adjoint and co-adjoint operators:
\begin{definition}
Let $G$ be a Lie group, $R_{h}$ (respectively $L_h$) denotes the right translation (respectively left translation) by $h \in G$, namely $R_h(g)\eqdef gh$ (respectively $L_h(g) \eqdef hg$). 
\\
Let $h\in G$, the adjoint operator $\Ad_h: G \times \mathfrak{g} \mapsto \mathfrak{g}$ is defined by
\begin{equation}
\Ad_h(v) \eqdef dL_h \cdot dR_{h^{-1}}(v)\,.
\end{equation}
Then, $\Ad_h^*$ is the adjoint of $\Ad_h$ defined by duality on $\mathfrak{g}$.
\\
Their corresponding differential map at $\on{Id}$ are respectively denoted by $\ad$ and $\ad^*$.
\end{definition}
Consider $\mathcal{L}:TG \mapsto \R$ a Lagrangian which satisfies the following invariance property,
\begin{equation}
\mathcal{L}(g,\dot{g}) = \mathcal{L}(\on{Id},dR_{g^{-1}}(\dot{g}))\,.
\end{equation}
The reduced Lagrangian is 
$\ell:\mathfrak{g}\mapsto \R$  defined by $\ell(v) = \mathcal{L}(\on{Id},v)$ for $v \in \mathfrak{g}$ and the variational problem can be rewritten as
\begin{equation}\label{Eq:FirstOrderSystem}
\inf \int_0^1 \ell(v) \ud t \quad \text{     subject to     } \quad
\begin{cases}
\dot{g} = dR_g(v) \\
g(0)=g_0 \in G \text{ and } g(1)=g_1 \in G\,.
\end{cases}
\end{equation}

In order to compute the Euler-Lagrange equation for \eqref{Eq:FirstOrderSystem}, one needs to compute the variation of $v$ in terms of the variation of $g$. It is given by $\dot{w} - \ad_{v}w$ for any path $w(t) \in T_{\on{Id}}G$. Therefore, the Euler-Lagrange equation reads
\begin{equation}\label{Eq:EPDiff}
\left(\partial_t + \ad^*_v\right) \frac{\delta \ell}{\delta v} = 0\,.
\end{equation}
When the Lagrangian is a kinetic energy, one has $\ell(v) = \frac{1}{2}\langle v, Lv\rangle$, which will be also denoted by $\frac 12 \| v \|_\mathfrak{g}^2$, where $L:\mathfrak{g} \mapsto \mathfrak{g}$ is a quadratic form and $ \langle \cdot , \cdot \rangle$ denotes the dual pairing, one has $\frac{\delta \ell}{\delta v}=Lv$ and $Lv$ is the so-called momentum. Then, the critical curves are determined by their initial condition $(g(0),\dot{g}(0))$
and the Euler-Poincaré equation \eqref{Eq:EPDiff}, together with the flow equation $\dot{g} = dR_g(v)$.
\par
Note that our situation differs with the incompressible Euler equation since we work with metrics that are strong as explained in the next paragraph.
\subsection{Geodesics for strong right-invariant metrics}\label{Sec:StrongInvMetrics}
In general, the Lie group structure is no longer present in infinite dimensions in situations of interest, as shown in \cite{Omori1978} and the notion of Lie algebra usually does not make sense. The group of diffeomorphisms is often only a topological group and a smooth Riemannian manifold and thus $\mathfrak{g}$ has to be understood as $T_{\on{Id}}G$. In the case of the group of diffeomorphisms, the key point is the loss of regularity of left composition\footnote{Left multiplication is not smooth on $H^s$ but right composition is smooth because it is linear.}, even if the underlying topology is that of $H^s$ for $s$ big enough. This loss of derivatives clearly appears in Equation \eqref{Eq:EPDiff}.
Therefore, to make sense of the above Euler-Poincaré equation, one often needs to work with two different topologies, so that right composition is a sufficiently regular map \cite{Ebin1970,Constantin2003}.
\par
An important contribution of \cite{Ebin1970} among others, is that the loss of smoothness can be circumvented by switching from Eulerian to Lagrangian coordinates. Namely, the incompressible Euler equation are interpreted as ordinary differential equations on a Hilbert space, which enables to conclude to local well poshness. Thus, the analytical study of these systems is sometimes better suited in Lagrangian coordinates.
Moreover, for the Euler equation, two topologies are required since the Lagrangian is the $L^2$ norm of the vector field and one wants to work on a group of diffeomorphisms whose underlying topology is stronger. 
This strategy works because the Euler-Poincaré equation \eqref{Eq:EPDiff} preserves the smoothness of the initial data as shown in \cite{Ebin1970}.
\par
Also remarked in \cite{Ebin1970,MisolekPreston2010}, in the case of the kinetic energy for an $H^s$ norm where $s>d/2+1$, a unique topology is enough to work with. Such a norm defines a strong Riemannian metric, in the sense of \cite{Ebin1970} and also \cite[Theorem 4.1]{MisolekPreston2010}, on the group of diffeomorphisms. Let us underline that the right-invariant metric $H^s$ for $s>d/2+1$ is indeed a smooth metric and one can apply standard results such as the Gauss lemma valid in infinite dimensions, as shown in \cite{Lang1999}. This smoothness result is also valid for fractional order Sobolev metrics \cite{Bauer20152010}. Moreover, using standard methods of calculus of variations, completeness results have been recently established in \cite{BruverisVialard}: the group is metrically complete which implies that geodesics can be extended for all time. Last, any two diffeomorphisms in the connected component of identity can be joined by a minimizing geodesic.
\subsection{The higher-order case}

In \cite{HOSplines1}, higher-order models that are also invariant are proposed on groups of diffeomorphisms but for the standard Riemannian cubics functional, no analytical study was provided. However, the formulation of the Euler-Lagrange equation in reduced coordinates is so simple that it is worth summarizing some of the results in \cite{HOSplines1}. Namely, let $\mathcal{L}:T^kG \mapsto \R$ be a Lagrangian defined on the $k^{\text{th}}$ order tangent bundle, then a curve
$g: [t_0, t_1] \rightarrow G$
is a critical curve of the action
\begin{equation}\label{Eq:Euler-Lagrange_action}
 \mathcal{J}[q] = \int_{t_0}^{t _1} \mathcal{L} \left( g(t), \dot g(t),....,  g^{ (k) }(t) \right)  dt  
\end{equation}
among all curves  $g(t) \in G$ whose first $(k-1)$ derivatives are fixed at the endpoints: $g^{(j)}(t_i)$, $i=0,1$, $j=0,..., k-1$,
if and only if $g(t)$ is a solution of the $k^{th}$-order Euler-Lagrange equations
  \begin{equation}\label{EL-eqns}
    \sum_{j=0}^k (-1)^j \frac{d^j}{dt^j}  \frac{\partial  \mathcal{L}}{\partial  g^{ (j) }}=0\,.
  \end{equation}
Now, an invariant higher-order Lagrangian is completely defined by its restriction on the higher-order tangent space at identity. As a consequence, the Lagrangian \eqref{Eq:Euler-Lagrange_action} can be rewritten as 
\begin{equation*}\label{Reduced_curve}
\mathcal{L}\left([g]_{g(t_0)}^{(k)}\right) = \ell\left(v(t_0), \dot{v}(t_0), \ldots , v^{(k-1)}(t_0)\right),
\end{equation*}
where $v \eqdef dR_{g^{-1}}(\dot{g})$ as detailed in \cite{HOSplines1}. The corresponding higher-order Euler-Poincaré equation is 
\begin{equation}\label{EP_k_dash}
\left(\partial_t  + \operatorname{ad}^*_v \right)  
\sum _{ j=0}^ {k-1} (-1)^j \partial _t ^j  \frac{\delta \ell }{\delta  v^{ (j) } }=0\,.
\end{equation}

Let us instantiate it in the case of the Lagrangian \eqref{SplinesFunctional} for which the previous setting applies.
Indeed, in the case of a Lie group $G$  with a right-invariant metric, the covariant derivative can be written as follows:
Let $V(t) \in T_{g(t)}G$ be a vector field along a curve $g(t) \in G$
 \begin{equation}\label{Covariant_derivative_vf_prop}
      \frac{D}{Dt} V =\Big( \dot{\nu} + \frac{1}{2} \operatorname{ad}^\dagger_\xi \nu + \frac{1}{2} \operatorname{ad}^\dagger_\nu \xi - \frac{1}{2} [\xi, \nu]\Big) _G(g)\,,
    \end{equation}
    where $\operatorname{ad}^\dagger$ is the metric adjoint defined by
     \begin{equation*}\label{ad-dagger-def}
 \operatorname{ad}^\dagger_{\nu}{\kappa} := 
 (\operatorname{ad}^*_\nu (\kappa^\flat))^\sharp
 \end{equation*}
 for any $\nu, \kappa\in \mathfrak{g}$ where $\flat$ is the isomorphism associated with the metric from$\mathfrak{g}$ to $\mathfrak{g}^*$ and $\sharp$ is its inverse. They correspond to raising and lowering indices in tensor notation.
 \par
Therefore, the reduced lagrangian for \eqref{SplinesFunctional} is 
\begin{equation}\label{ReducedSplinesFunctional}
\mathcal{J}(x) = \int_0^1 \| \dot{\xi} + \ad^\dagger_\xi \xi\|^2_\mathfrak{g} \ud t\,.
\end{equation}
 For this Lagrangian, the Euler-Lagrange equation \eqref{EP_k_dash} reads
 \begin{equation}
\label{2nd-EPeqns2}
\left(\partial _t + \operatorname{ad}^\dagger_ \xi \right) \left(\partial _t \eta
 +  \operatorname{ad}^\dagger_{ \eta} \xi  + \operatorname{ad}_{\eta} \xi \right) =0
 \quad \hbox{where}\quad
 \eta := \dot{ \xi } + \operatorname{ad}^\dagger _ \xi \xi.
\end{equation}
While this formula is compact, it is a formal calculation in the case of the group of diffeomorphisms since there is a loss of smoothness which is already present in the acceleration formula \eqref{2nd-EPeqns2}. This formula also resembles closely to the Jacobi field equations for such metrics which was used in \cite{MisolekPreston2010} and for which an integral formulation is available \cite[Proposition 5.5]{MisolekPreston2010}. This is of course expected due to Formula \eqref{Eq:EulerLagrangeClassical} below.

As mentioned in Section \ref{Sec:StrongInvMetrics}, there is a clear obstacle to use reduction since the operator $ \ad^\dagger$ is unbounded on the tangent space at identity due to the loss of derivative. 
Following \cite{Ebin1970}, one can use the smooth Riemannian structure on $\mc D^s$ to check that the functional \eqref{SplinesFunctional} is well defined.
The following proposition of \cite{splinesanalyse} is valid in infinite dimensions:
\begin{proposition}
Let $(M,g)$ be an infinite dimensional Riemannian manifold and $$\Omega_{0,1}(M) := \{  x \in H^2([0,1],M) \, | \, x(i) = x_i \,,\, \dot{x}(i) = v_i \text{ for } i=0,1 \}$$ be the space of paths with first order boundary constraints for given $(x_0,v_0) \in TM$ and $(x_1,v_1) \in TM$. 
\noindent
The functional \eqref{SplinesFunctional} is smooth on $\Omega_{0,1}(M)$ and 
\begin{equation*}
\mathcal{J}'(x)(v) = \int_0^1 g\left(\frac{D^2}{Dt^2}\dot{v},\frac{D}{Dt}\dot{x}\right) - g\left(R\left(\dot{x},\frac{D}{Dt}\dot{x}\right),v\right) \, \ud t\,.
\end{equation*}
A critical point of $\mathcal{J}$ is a smooth curve that satisfies the Riemannian cubic equation
\begin{equation}\label{Eq:EulerLagrangeClassical}
\frac{D^3}{Dt^3}\dot{x} - R\left(\dot{x},\frac{D}{Dt}\dot{x}\right)\dot{x} = 0\,.
\end{equation}
\end{proposition}
The critical points of $\mathcal{J}$ are called Riemannian cubics or cubic polynomials. In Euclidean space, the curvature tensor vanishes and one recovers standard cubic polynomials.
In this paper, we will be interested in existence of minimizers for the functional \eqref{SplinesFunctional} in the case of the group of diffeomorphisms endowed with a strong right-invariant metric. 
The existence of minimizers (and the fact that $\mathcal{J}$ satisfies the Palais-Smale condition) does not follow from the corresponding proof in \cite{splinesanalyse} since it strongly relies on the finite dimension hypothesis and compactness of balls. Moreover, as shown above, it is not possible to follow the proof of \cite{BruverisVialard} since the reduced functional \eqref{ReducedSplinesFunctional} is not well defined on the tangent space at identity. Therefore, we will write in Section \ref{Sec:GeodesicEquation} the variational problem in Lagrangian coordinates so that we take advantage of the smoothness of the metric.

\section{The main result and the strategy of proof}

The smoothness of the metric is not enough to deal with the problem of existence of minimizing geodesics and the well-known example is the work of Brenier on generalized solutions of Euler equation \cite{BrenierLeastAction}. As explained above, a technical important difference is that the $L^2$ metric on the space of diffeomorphisms is a weak metric in the sense of Ebin and Marsden \cite{Ebin1970}, whereas we work with a strong metric. The group of diffeomorphisms endowed with a right-invariant Sobolev metric of order $s > d/2+1$ is complete in all sense of the Hopf-Rinow theorem as proven in \cite{BruverisVialard}. Passing to second-order derivatives has been less treated from a variational point of view, and to the best of our knowledge it has never been addressed in the case of right-invariant norms on the group of diffeomorphisms. \par
We present hereafter the three main steps developed in the paper. 
In Section \ref{Sec:GeodesicEquation}, the first step is to choose a convenient formulation of the acceleration which is done .
The first technical choice follows Ebin and Marsden \cite{Ebin1970} and it consists in writing the acceleration in Lagrangian coordinates instead of Eulerian coordinates. The point is to avoid the loss of smoothness of the Eulerian formulation. The second choice which appears the most important from an analytical point of view consists in using the second derivative of the diffeomorphism as the main variable to compute the geodesic equation. We therefore work on $H^2_0([0,1])$ in order to avoid boundary terms. At this step, we strongly use the one dimensional setting. This simple change of variable leads to geodesic equation that have a Hamiltonian formulation enjoying important analytical properties. Let us give an overview of the new set of equations. Now the variable $q$ represents a function in $L^2([0,1])$ and thus the dual space can be identified with $L^2([0,1])$, we have the following formulation, with $\mathcal{H}$ being the Hamiltonian
\begin{equation}
\begin{cases}
\dot{q} = \partial_p \mathcal{H}(p,q) = K(q)(p)\\
\dot{p} = -\partial_q \mathcal{H}(p,q)= -B(q)(p,p)\,,
\end{cases}
\end{equation}
where $K(q)$ is a bounded linear operator on $L^2([0,1])$ which is continuous w.r.t. $q$ for the weak topology. As is well-known, $K(q)$ is the inverse of the metric tensor. The  operator $B(q)$ is bounded as a bilinear operator on $L^2$ and is continuous w.r.t. $q$ for the weak topology. However, it is not continuous w.r.t. to $p$ for the weak topology due to the bilinear structure. 
\par Importantly, the operator $B$ is non-local and therefore, the acceleration functional \eqref{Eq:ShortAccelerationFunctional} written below is not the integral of a Caratheodory type integrand. This non-local term is more precisely $\mathcal{U}(p^2,q)$ where $\mathcal{U}$ is defined in Formula \eqref{Eq:NonLocalOperator}. In such cases, there does not exist a general theory of relaxation and the relaxation formulation has to be studied directly.
\par
Thus, the second step consists in studying the relaxation of the acceleration functional that can be written as follows
\begin{equation}\label{Eq:ShortAccelerationFunctional}
\mathcal{J}(p,q) = \int_0^1\|K(q)^{1/2}(\dot{p} + B(q)(p,p)) \|^2_{L^2} \ud t+ P(p(1),q(1))
\end{equation}
where $P$ is a relaxation of the endpoint constraint at time $1$ which is lower continous for the weak topology. 
Expanding the quadratic term, we have to deal with the weak limit of $\langle \dot{p},K(q)\dot{p}\rangle$ denoted by $\nu$ and the weak limit of $B(q)(p,p)$ which only involves the weak limit of $p^2$ denoted by $\mu$. These two weak limits are related to each other. The relation is given by the following inequality
\begin{equation}\label{Eq:SquareRoot}
\left(\partial_t \sqrt{\mu}\right)^2 \leq \nu \,,
\end{equation}
for which a careful analysis is developed in section \ref{Sec:FRFunctional}. In fact, Equation \eqref{Eq:SquareRoot} can be made rigorous using the Fisher-Rao functional which is a convex functional on measures on the time space domain $D = [0,1]^2$ defined by 
\begin{equation*}
\on{FR}_f(\mu,\nu) = \int_D \frac14 \frac{\rho_\nu^2}{\rho_\mu} f \ud \lambda\,
\end{equation*}
where $f$ is a positive and continuous weight function on $D$ and $\rho_\nu$ and $\rho_\mu$ are the densities of $\mu$ and $\nu$ with respect to a dominating measure $\lambda$. 
Now, formula \eqref{Eq:SquareRoot} can be rewritten as
\begin{equation}
 \on{FR}(\mu,\partial_t\mu) \leq \nu \,.
 \end{equation}
as linear operators on continuous positive functions $f$ on the domain $D$. Therefore, the relaxation of $\mathcal{J}$ will make appear the Fisher-Rao term $ \on{FR}(\mu,\partial_t\mu)$. Let us underline that informally, this is the cost associated with the oscillations that are generated on $p$. However, in order to prove that the relaxation of the functional exactly involves this quantity we have to construct explicitly the oscillations that generate the measure $\mu$ at the given cost $ \on{FR}(\mu,\partial_t\mu)$. This is a technical step that relies on the construction of solutions to the first equation of the Hamiltonian system and also on an explicit construction of the oscillations.
\par
The last step in Section \ref{Sec:MainTheorem} consists in a standard analysis of optimality conditions by means of convex analysis. In particular, the Euler-Lagrange equation associated with the Fisher-Rao functional is a Riccati equation.

\section{Geodesic equations in Hamiltonian coordinates}\label{Sec:GeodesicEquation}

\subsection{The group of diffeomorphisms and its right-invariant metric}
We consider the space of $H_0^2([0,1])$ of Sobolev functions of order 2 on the real interval $[0,1]$ with vanishing Dirichlet boundary condition on the function and its first derivative. We also define the group $\on{Diff}_0^2([0,1])$ of Sobolev diffeomorphisms, as follows:

\begin{definition}\label{def:GroupOfDiffeos}
The group of Sobolev diffeomorphisms $\on{Diff}_0^2([0,1])$ on the real interval $[0,1]$ is
\begin{equation}\label{Eq:GroupOfDiffeos}
\on{Diff}_0^2([0,1]) = \{ \on{Id} + f \, | \, f \in   H_0^2([0,1]) \text{ and } 1+f'(x) > 0 \} \,.
\end{equation}
\end{definition}
Note that this group is actually the connected component of identity which justifies the subscript.
 On this group, we define the right-invariant metric by defining it on the tangent space at the identity $\on{Id}$. Note that elements of the tangent space at identity will sometimes be called vector fields by analogy with fluid mechanic.
 \begin{definition}\label{def:MetricOnTheGroup}
 Let $u,v$ be two tangent vectors at $\on{Id} \in \on{Diff}_0^2([0,1])$, define
\begin{equation}\label{eq:MetricAtIdentity}
G_{\on{Id}}(u,v) = \langle u, v \rangle_{H_0^2} \eqdef \int_{0}^1 \partial_{xx}u\,\partial_{xx}v\ud x\,.
\end{equation}
The scalar product on $T_{\on{Id}}\on{Diff}_0^2([0,1])$ completely defines the right-invariant metric on $\on{Diff}_0^2([0,1])$: let $X_\varphi,Y_\varphi$ be two tangent vectors at $\varphi \in \on{Diff}_0^2([0,1])$, then
\begin{equation}\label{eq:MetricH2}
G_{\varphi}(X_\varphi,Y_\varphi) = \langle X_\varphi \circ \varphi^{-1}, Y_\varphi \circ \varphi^{-1} \rangle_{H_0^2} \,.
\end{equation}
\end{definition}
Remark that due to boundary conditions in $ H_0^2([0,1])$, the metric is non degenerate.

\subsection{Existence of minimizing geodesics and their formulation in Eulerian coordinates}

We first begin with an important lemma that will be used in some other sections.

\begin{lemma}
Let $v \in L^2([0,1],H_0^2([0,1])$ be a time dependent vector field, then there exists a unique solution to the flow equation:
\begin{equation}\label{eq:ODE}
\begin{cases}
\partial_t \varphi(t,x) = v(t,\varphi(t,x))\\
\varphi(0)= \on{Id}\,.
\end{cases}
\end{equation}
and $\varphi \in C^0([0,1],H^2([0,1]))$.
\end{lemma}

\begin{proof}
See \cite[Section 4]{BruverisVialard}.
\end{proof}

\begin{lemma}
Let $v_n \in L^2([0,1],H_0^2([0,1])$ a weakly converging sequence, then $\varphi_n$ converges uniformly on $[0,1]  \times [0,1]$.
\end{lemma}
\begin{proof}
See \cite[Lemma 7.1]{BruverisVialard}.
\end{proof}
\begin{theorem}
Let $\varphi_0,\varphi_1 \in \on{Diff}_0^2([0,1])$, there exists a minimizing geodesic between $\varphi_0$ and $\varphi_1$ in $H^1([0,1],\on{Diff}_0^2([0,1]))$.
\end{theorem}
\begin{proof}
See \cite[Theorem 7.2]{BruverisVialard}.
\end{proof}

The Euler-Poincaré equation has a strong sense when the initial condition is sufficiently smooth due to the propagation of regularity \cite[Theorem 12.1]{Ebin1970}: If $m(t) = Lv(t)$, one has
\begin{equation}\label{Eq:EPDiffSimple}
\partial_t m + \ad^*_{Km}m = 0\,,
\end{equation}
where $K \eqdef L^{-1}$ is the inverse of the differential operator $L$. The operator $K$ is defined by the reproducing kernel $k$ of the Hilbert space $H^2_0([0,1],\R)$ for which the following formula holds
\begin{equation}\label{Eq:Kernel}
k(s,t) = k_0(s,t) +\left(-1 + \frac{1}{2}(s+t) - \frac{1}{3}ts\right)(ts)^2 \,,
\end{equation}
where $k_0$ is defined by
 \begin{equation}
k_0(s,t) = 
\begin{cases}
1 + st + \frac{1}{2} ts^2-\frac 16 s^3 \text{ if } s<t \\
1 + st + \frac{1}{2} st^2-\frac 16 t^3 \text{ otherwise}\,.
\end{cases}
\end{equation}
This kernel is not a usual one in the literature and therefore we prove this formula in Appendix \ref{Ap:rkhs}.\\
Then, $K: L^2([0,1],\R) \mapsto L^2([0,1],\R)$ is defined by 
$K(f)(s) = \int_0^1 k(s,t) f(t) \ud t$.
\par
An important case which has been used a lot in applications is when the initial momentum is a sum of Dirac masses. We detail it now since we will need an explicit example of such geodesics in Section \ref{Sec:MainTheorem}.
In imaging applications, this case is known as the landmark manifold \cite{MicheliCurvatureLandmarks}, which is parametrized by, for any $n \in \N^*$, $$\mathcal{M}_n \eqdef \left\{ (q_1,\ldots,q_n) \in [0,1]^n \, ; \, q_i \neq q_j \text{ for } i\neq j \right\}\,.$$ 
Consider $m(0) = \sum_{i=1}^n p_i(0)\delta_{q_i(0)}$, then Equation \eqref{Eq:EPDiffSimple} can be rewritten only in terms of $p_i(t) \in \R$ and $q_i(t) \in [0,1]$ (see \cite[Proposition 6]{MicheliCurvatureLandmarks} for a detailed account). It is actually the Hamiltonian formulation of the geodesic equations of the so-called manifold of landmarks. Namely, one has
\begin{equation}\label{Eq:HamiltonianLandmarks}
\begin{cases}
\dot{q}_i = \sum_{j=1}^n k(q_i,q_j)p_j\\
\dot{p}_i = -\sum_{j=1}^n p_i  \partial_1 k(q_i,q_j)p_j
\end{cases}
\end{equation}
where $\partial_1$ denotes the partial derivative with respect to the first variable.
The corresponding Hamiltonian is
\begin{equation}\label{Eq:HamiltonianLand}\mathcal{H}(p,q) \eqdef \frac12 \sum_{i,j=1}^n p_i k(q_i,q_j)p_j\,.\end{equation}
Endowed with the co-metric \eqref{Eq:HamiltonianLand}, the manifold of landmarks $\mathcal{M}_n$ is a complete metric space for our kernel of interest \eqref{Eq:Kernel}.
Geodesic completeness \cite{BruverisVialard} implies that solutions of the Hamiltonian equations \eqref{Eq:HamiltonianLandmarks} are defined for all time. We prove in Appendix \ref{Ap:Proofs} the following proposition.

\begin{proposition}\label{Th:ExampleOfGeodesic}
There exists a geodesic on the group of diffeomorphisms $\on{Diff}_0^2([0,1])$ which satisfies
 the following properties: $\varphi(t,\frac 12) = \frac 12$, $\partial_t \varphi_x(t,1/2)$ is decreasing and $\lim_{t \to \infty} \varphi_x(t,1/2) = 0$.
\end{proposition}

\subsection{Hamiltonian formulation of geodesic equations in Lagrangian coordinates}
Definition \ref{def:MetricOnTheGroup} is the natural definition of a right-invariant metric on a group. However, in such a form, the smoothness of the metric is not obvious. As proven in \cite{Ebin1970}, the way to prove it consists in switching from Eulerian to Lagrangian coordinates.

\begin{proposition}
The metric $G$ is a smooth Riemannian metric on $\on{Diff}_0^2([0,1])$.
\end{proposition}
\begin{proof}
We will use the following identities,
\begin{align*}
\partial_x\left(\ph\i\right) &= \frac 1{\partial_{x}\ph \o \ph\i} \\
\partial_{xx}\left(\ph\i\right) &= \frac {\partial_{xx}\ph \o\ph\i}{\left(\partial_x\ph \o \ph\i\right)^3}\,.
\end{align*}
Note that, for a given function $X \in H^2_0(\R,\R)$ we have 
\begin{align*}
\partial_{x}(X \circ \varphi\i) = \left( \frac {\partial_{x}X}{\partial_{x}\ph} \right) \circ \ph\i \\
\partial_{xx}(X \circ \varphi\i) = \left( \frac {\partial_{x}\left(\frac{\partial_{x}X}{\partial_{x}\ph}\right)}{\partial_{x}\ph} \right) \circ \ph\i \,.
\end{align*}
And now the metric itself,
\begin{align*}
G_{\ph}(X, Y) &= \langle X\o\ph\i, Y\o\ph\i\rangle_{H^2} \\
&= \int_0^1  \frac{1}{\partial_{x}\ph} \left[  \partial_{x}\left( \frac{\partial_{x} X}{\partial_{x}\ph}\right) \right]^2 \, \ud x\,,
\end{align*}
We consider the metric $G$ as a mapping (with a little abuse of notations)
\[
G : \on{Diff}_0^2([0,1]) \to \mathcal{L}(H_0^2([0,1]), H_0^2([0,1]))\,,
\]
defined by the following relation, for all $X,Y \in H_0^2([0,1])$
\[
G_\ph(X,Y) = \langle G(\ph).X, Y \rangle_{H_0^{-2} \x H_0^2}\,.
\]
Since $H_0^2([0,1])$ and $H_0^1([0,1])$ are Hilbert algebras, polynomial functions on these domains are smooth and therefore the metric itself is smooth.
\end{proof}

We detailed the proof for two reasons: (1) these formulas will be used later on and (2) for ease of understanding. However, this proof is a particular case of \cite{Ebin1970}, where it is proved that the standard right-invariant $H^s$ Sobolev metric on $\on{Diff}^s(M)$ for $s>\on{dim}(M)/2+1$ is smooth.

\begin{corollary}
The following formula for the metric $G$ holds:
For every $X\in H_0^2([0,1])$ and $\varphi \in \on{Diff}_0^2([0,1])$ one has
\begin{equation}
G_\ph(X,X) = \int_0^1  \frac{1}{\partial_{x}\ph} \left[  \partial_{x}\left( \frac{\partial_{x}X}{\partial_{x}\ph}\right) \right]^2 \, \ud x\,.
\end{equation}

\end{corollary}

In order to simplify the next computations, we will use the following change of variable which will appear as a key point in the treatment of the variational problem of Riemannian cubics.
\begin{proposition}\label{Eq:PropositionChangeOfVariable}
Under the change of variable $q\eqdef \partial_{x}(\log(\partial_{x} \ph))$, one has 
 $\varphi \in \on{Diff}_0^2([0,1])$ is equivalent to $q \in L^2([0,1])$ satisfying the constraints 
\begin{align} 
&\int_0^1 q(x) \ud x = 0 \,,\label{eq:FirstConstraint}\\
& \int_0^1 e^{\int_0^x q(u) \ud u} \ud x= 1\,.\label{eq:SecondConstraint}
 \end{align}
\end{proposition}

\begin{proof}
Using the formula \eqref{def:GroupOfDiffeos}, these are direct computations but note that the boundary constraint  $\varphi'(1)=1$ implies that $\int_0^1 q(x) \ud x = 0 $ and the constraint $\varphi(1)=1$ can be written as $$\int_0^1 e^{\int_0^x q(u) \ud u} \ud x = 1\,.$$
\end{proof}

Note that the second constraint \eqref{eq:SecondConstraint} is nonlinear in $q$. We now introduce the following notations:

\begin{notation}
For a given $q \in L^2([0,1])$, we will use the following maps 
$$\eta: L^2([0,1]) \to H^1([0,1]) \quad \text{ and } \quad \varphi:  L^2([0,1]) \to H^2([0,1])$$ defined by
\begin{align}
& \eta(q)(x) = \exp\left(\int_0^x q(u) \ud u\right)\,,\\
& \varphi(q)(x) = \int_0^x \eta(q)(y) \ud y\,.
\end{align}
In order to alleviate notations, we will omit the argument $q$ in $\eta(q)$ and $\varphi(q)$ when no confusion is possible.\end{notation}

\begin{proposition}
Under this change of variable, the group $\on{Diff}_0^2([0,1])$ is the Hilbert submanifold $Q$ of $L^2([0,1])$ defined by the constraints \eqref{eq:FirstConstraint} and \eqref{eq:SecondConstraint}. 
\\
The tangent space at $q \in Q$ is 
\begin{equation}\label{eq:TangentSpace}
T_qQ = \left\{ X \in L^2([0,1],\R) \, : \, \int_0^1 X(x) \ud x = 0 \text{ and }  \int_0^1 X(x) \varphi(x) \ud x = 0 \right\}\,.
\end{equation}
\\
The metric reads, for $X \in T_qQ$
\begin{equation}\label{Metric}
G(q)(X,X) = \int_0^1\frac {X^2} {\eta(q)} \, \ud x\,.
\end{equation}
We denote the scalar product $G$ by $\langle \cdot , \cdot \rangle_{1/\eta}$.
\end{proposition}

\begin{proof}
We have to determine explicitly the tangent space. The first constraint \eqref{eq:FirstConstraint} is linear and this obviously leads to the constraint on a vector of the tangent space $X$, $\int_0^1 X(x) \ud x = 0$. The second constraint \eqref{eq:SecondConstraint} can be differentiated as follows
\begin{equation}
\int_0^1 \left(\int_0^x X(u) \ud u \right) \, \eta(q)(x) \ud x\,,
\end{equation} 
and it gives, by integration by parts, $\int_0^1 X(x) \varphi(x) \ud x = 0$ since $\int_0^1 X(x) \ud x = 0$.
\par 
To obtain the expression of the metric \eqref{Metric}, it suffices to remark that 
\begin{equation}
\delta q =  \delta \partial_{x}(\log(\partial_{x}\ph)) =  \partial_{x}(\delta \log(\partial_{x}\ph)) = \partial_{x}\left(\frac{\partial_{x}(\delta \varphi)}{\partial_{x}\varphi}\right)
\end{equation}
where we denoted by $\delta q$ and $\delta \varphi$ variations of, respectively, $q$ and $\varphi$. 
\end{proof}
We aim at computing the geodesics for the Riemannian metric above on the submanifold $Q$ for the metric $G$.
We need the variations of $\eta(q)$ with respect to $q$ in the ambient space $L^2$.
For $f \in L^2([0,1],\R)$, one has
$$ D\eta(q)(f)(x) =\eta(q)(x) \int_0^x f(y)\ud y  \,,$$
and its adjoint reads:
$$ D\eta(q)^*(f)(x) = \int_x^1 \eta(q)(y) f(y)\ud y  \,.$$

Since the geodesic equations are equations on the constrained submanifold $Q$, we need the formula of the orthogonal projection on the tangent space to the submanifold.
\begin{proposition}
Let $\pi_q$ be the orthogonal projection onto the tangent space $T_qQ$ at point $q$ for the metric $G$. 
One has 
\begin{equation}\label{eq:projection}
\pi_q(f) = f -  \begin{bmatrix} \eta & \varphi \eta \end{bmatrix} H_2^{-1}  \begin{bmatrix} \langle f , \eta \rangle_{1/\eta}  \\ \langle f , \varphi \eta \rangle_{1/\eta}  \end{bmatrix}
\end{equation}
where $H_2$ is the Hilbert matrix given by $H_2 = \begin{bmatrix} 1 &  1/2 \\  1/2 &  1/3 \end{bmatrix}$.
\\
Denote by $\pi^*_q:T^*Q \to T^*Q$ the adjoint of $\pi_q$ w.r.t. the $L^2$ scalar product. One has
\begin{equation}\label{eq:projectionDuale}
\pi_q^*(p) = p -  \begin{bmatrix} 1 & \varphi \end{bmatrix} H_2^{-1}  \begin{bmatrix} \langle p , 1 \rangle_{\eta}  \\ \langle p , \varphi  \rangle_{\eta}  \end{bmatrix}\,.
\end{equation}
\end{proposition}

\begin{proof}
The two constraints can be written as
$
\langle f, \eta  \rangle_{1/\eta}= 0$ and $
\langle f, \eta \varphi  \rangle_{1/\eta} = 0$ and  Formula \eqref{eq:projection} follows. The Gram matrix of the vectors $1$ and $\varphi$ associated with the scalar product $\langle \cdot , \cdot \rangle_\eta$ is $H_2$. Indeed, we have $\int_0^1 \eta \ud x = 1$ and
\begin{equation*}
\int_0^1 \eta \, \varphi \ud x = \left[ \frac 12 \varphi^2 \right]_0^1 = \frac 12
\, \mbox{} \text{ and } \,\int_0^1 \eta \, \varphi^2 \ud x = \left[ \frac 13 \varphi^3 \right]_0^1 = \frac 13\,,
\end{equation*}
since $\partial_x \varphi = \eta$.
\end{proof}

\begin{theorem}\label{Th:GeodesicEquations}
The geodesic equation on $Q$ for the metric $G$ is, in its Hamiltonian formulation, the first-order system on $T^*Q$ defined by
\begin{equation}\label{FirstEquationSystem2}
\begin{cases}
\dot{q} = \eta(q)p\\
\dot{p} = - \frac 12 \int_x^1 \eta(q)p^2 \ud y +  \begin{bmatrix} 1 & \varphi \end{bmatrix} H_2^{-1}  \begin{bmatrix} a \\ b +c \end{bmatrix}
\end{cases}
\end{equation}
where the coefficients $a,b,c$ depend on $p,q$ and are given by:
\begin{subequations}
\begin{align}
&a = \frac 12 \int_0^1 x \, p^2 \circ \varphi^{-1} \ud x \,,\label{eq:FirstProjection}\\
&b = \frac 32 \int_0^1 \left( \int_0^x p \circ \varphi^{-1}  \right)^2 \ud x \, , \label{eq:SecondProjection1}\\
&c = \frac 14 \int_0^1 x^2 \, p^2 \circ \varphi^{-1} \ud x\,.\label{eq:SecondProjection2}
\end{align}
\end{subequations}
\end{theorem}
We postpone the proof of the theorem in appendix \ref{Ap:ComputationOfGeodesicEquations} since it is based on lengthy but rather straighforward computations.
Note that in the first equation of the geodesic equations as stated in \eqref{FirstEquationSystem2}, the constraints are implicit: For an intial $(p,q) \in T^*Q$, the local solution to the geodesic equation will stay in $T^*Q$. 
Moreover, this Hamiltonian formulation enables to retrieve easily in our context a similar result presented in \cite{Kolev1} about regularity of geodesics on the group of diffeomorphisms of the circle as explained below.
\begin{corollary}\label{Th:Regularity}
Let $\varphi_0= \on{Id}$ and $\varphi_1 \in H^n$ with $n\geq 3$ be two diffeomorphisms in $\on{Diff}^2_0([0,1])$.
Then, for every geodesic joining $\varphi_0$ and $\varphi_1$, the initial tangent vector lies in $H^n([0,1])$.
\end{corollary}

\begin{proof}
In this proof, the spatial notation will be omitted and $p,q,\eta(q),\ldots$ will denote time dependent mappings with values in spaces of functions that depend on $x$, the spatial variable. Let us recall the geodesic equations 
\begin{subequations}
\begin{align}
&\dot{q} = \eta(q)p\,,\label{eq:qEq}\\
& \dot{p} = - \frac 12 \int_x^1 \eta(q)p^2 \ud y +  \begin{bmatrix} 1 & \varphi \end{bmatrix} H_2^{-1}  \begin{bmatrix} a \\ b +c \end{bmatrix} \, . \label{Eq:MomentumEquation}
\end{align}
\end{subequations}
The key observation is that the momentum equation \eqref{Eq:MomentumEquation} has more regularity on its right-hand side due to the integral operator. The hypothesis implies $q(0),q(1) \in H^{n-2}$. We prove the result by induction on $n$ and we first assume that $n=3$. 
Since $(p,q) \in L^2([0,1]^2)$ is a solution of the geodesic equations above, $p\in C^0([0,1],L^2([0,1]))$ and therefore $p^2 \in C^0([0,1],L^1([0,1])$ and since $\eta(q) \in C^0([0,1])$, we have that $\dot{p} \in C^0([0,1],W^{1,1}([0,1]))$. Now, observe that $p(t) = p(0) + \int_0^t \dot{p}(s) \ud s$ and thus 
\begin{equation*}
q(t) = q(0) + \left( \int_0^t \eta(q)(s) \ud s \right)\,p(0) + \int_0^t \eta(q) \int_0^s \dot{p}(u) \ud u \ud s\,.
\end{equation*}
This formula gives $p(0)$ in terms of $q(1)$ and $q(0)$.
\begin{equation}\label{Eq:FormulaBootstrap}
p(0) = \frac{1}{\int_0^1 \eta(q)(s) \ud s} \left(q(1) -q(0) -  \int_0^1 \eta(q) \left(\int_0^s \dot{p}(u) \ud u\right) \ud s \right)\,.
\end{equation}
The function $x \mapsto \int_0^1 \eta(q)(s,x) \ud s$ belongs to $H^1$ and is strictly positive. Thus, $\frac{1}{\int_0^1 \eta(q)(s) \ud s}$ belongs to $H^1$. As a consequence, we have that 
$p(0) \in W^{1,1}$ since $q(1) - q(0) \in H^1$ and the remaining term lies in $W^{1,1}$. 
Now, as $p$ is in $W^{1,1}$, the term $\int_x^1 \eta(q)p^2 \ud y$ is in $C^1$ and thus in $H^1$. Therefore, going back to formula \eqref{Eq:FormulaBootstrap}, we have that $p(0) \in H^1$, which implies that $\dot{\varphi}(0) \in H^3$.
\par
Let us assume that the result is proven for $n\geq 3$, and we prove it for $n+1$. We thus have $\eta(q) \in H^{n-1}$ and $p^2 \in H^{n-2}$ since $H^{n-2}$ is a Hilbert algebra. It implies that  $\dot{p} \in H^{n-1}$ by formula \eqref{Eq:MomentumEquation}. Therefore, formula \eqref{Eq:FormulaBootstrap} gives $p(0) \in H^{n-1}$ since $\int_0^t \eta(q)(s) \ud s$ in $H^{n-1}$ and proceeding as in the first step of the induction, we obtain $\dot{\varphi}(0) \in H^{n+1}$.
\end{proof}

Although the proof could be extended to other types of regularity such as fractional Sobolev index, our motivation consisted in showing potential applications of these Hamiltonian equations, which take advantage of the change of variable introduced in Proposition \ref{Eq:PropositionChangeOfVariable}.
\par
Since we will be interested in weak convergence, note that we can decompose the projection into two terms, one associated with $\dot{p}$ and one associated with $\frac 12 \int_x^1 \eta(q)p^2 \ud y$. The former is continuous with respect to the weak topology on $Q \subset L^2([0,1])$ whereas the latter is not.
\par
We will need to work on a more explicit representation of the solutions to the first equation. Indeed, in the first equation of system \eqref{FirstEquationSystem2} the constraint is implicit and we will make it explicit by introducing the projection $\pi^*_q$. Thus, we will be able to define solutions that will be useful to characterize the relaxation of the functional. Hereafter, we work again implicitly on time dependent variable with values in functional spaces. 
%
\begin{theorem}\label{th:Existence}
For a given path $p \in  L^2([0,1],L^\infty([0,1])$ and $q(0) \in L^\infty([0,1])$, there exists a unique solution $q \in H^1([0,1],L^\infty([0,1])$ to
\begin{equation}\label{ProjectedSystem}
\dot{q} = \eta(q) \pi^*_q(p)\,.
\end{equation}
\end{theorem}

The proof is given in Appendix \ref{Ap:Proofs}.

\begin{remark}
It might be surprising for the reader that we work with the space $L^2([0,1],L^\infty([0,1])$. Actually, we are not able to extend the previous result for the space $L^2([0,1],L^2([0,1])$, which is even probably wrong. What makes the proofs work with the sup norm is its invariance with respect to reparametrizations. In Eulerian coordinates, the system would be  well posed, however in Lagrangian coordinates, it is not true any more since the behavior of the right-invariant norm is not the same between Lagrangian and Eulerian coordinates.
\end{remark}

Since the weak topology will be studied, we present the key properties of the metric under weak convergence. Again, the proof of the following lemma is given in Appendix \ref{Ap:Proofs}.
\begin{lemma}\label{th:WeakConvergenceLemma}
If $q_n$ weakly converges to $q$ in $H^1([0,1],L^2([0,1])$, then 
\begin{enumerate}
\item $\eta(q_n)$ converges to $\eta(q)$ strongly in $(C^0(D),\| \cdot \|_\infty)$,
\item $\varphi(q_n)$ converges to $\varphi(q)$ strongly in $(C^0([0,1],C^1([0,1])),\| \cdot \|_{\infty,1})$ (the norm being the sup norm in time of the sup norm on $C^1$),
\item $\pi_{q_n}$ and $\pi_{q_n}^*$ strongly converge to, respectively, $\pi_{q}$ and $\pi_{q}^*$ as linear operators for the operator norm on $L^2$ in $C^0([0,1],\mathcal{L}(L^2([0,1])))$. 
\item If, in addition, $p_n$ weakly converges to $p$ in $H^1([0,1],L^2([0,1])$, then $\frac d {dt} \pi_{q_n} $ and $\frac d {dt} \pi_{q_n}^* $strongly converges as operators on $L^2$ in $C^0([0,1],\mathcal{L}(L^2([0,1])))$ (the norm being the sup norm in time of the operator norm).
\item Let $z_n \in H^1([0,1],L^2([0,1])$ be a weakly convergent sequence to $0$, then, under the above assumptions, we have $\pi_{q_n}(z_n) - z_n$ and  $\pi_{q_n}^*(z_n) - z_n$ strongly converge to $0$. Moreover,  $\left(\frac d {dt} \pi_{q_n}\right)(z_n)$ strongly converges to $0$.
\end{enumerate}

\end{lemma}

We will need the following result which will be proved in Appendix \ref{Ap:Proofs}.

\begin{proposition}\label{th:Boundedness}
Let $p_n \in L^\infty(D)$ be a bounded sequence and let $p_\infty \in L^\infty(D)$. Consider the solutions $q_n,q_\infty$ given by Proposition \ref{th:Existence} associated respectively with $p_n$ and $p_\infty$ for an initial condition $q(0) \in L^\infty([0,1])$.
\par
Then, if $p_n$ weakly converges in $L^2(D)$ to $p_\infty$ then $q_n$ weakly converges to $q_\infty$ in $H^1([0,1],L^2([0,1]))$.
\end{proposition}

\section{The Fisher-Rao functional and its main properties}\label{Sec:FRFunctional}
In order to study the minimization problem, we need to present the convex functional that will appear in the relaxation of the initial problem. This functional is well-known in areas such as information geometry and it is a particular example of a positively one-homogeneous convex functional on the space of measures. We collect below the properties needed for our study.

\begin{notation}
Let $\mu,\nu \in \mathcal{M}(D)$ be two measures that satisfy the following inequality
\begin{equation}\label{eq:FirstInequalityFormulation}
 \langle \mu , \partial_t f \rangle^2 \leq 4 \langle \nu , f \rangle  \langle \mu ,  f \rangle \,, 
\end{equation}
for every $f \in \mathcal{D}$ being a positive test function. 
This condition will be denoted by 
\begin{equation}\label{eq:InequalityCondition} | \partial_t \mu |  \leq 2\sqrt{\nu} \sqrt{\mu } \,.\end{equation}
\end{notation}

\begin{remark}
This notation $| \partial_t \mu |  \leq 2\sqrt{\nu} \sqrt{\mu }$ is coherent with the inequality obtained if $\mu$ and $\nu$ were $L^1(D)$ functions. This notation is also coherent with the following formula, under sufficient smoothness assumptions on $\mu$ and $\nu$,
\begin{equation*}
\left(\partial_t \sqrt{\mu}\right)^2 \leq \nu \,.
\end{equation*}
In the following, we rigorously define the meaning of this inequality.
\end{remark}

We now show that this inequality is stable under regularization by convolution.
\begin{lemma}\label{th:RegularisationOfFirstInequality}
Let $\mu,\nu \in \mathcal{M}(D)$ be two measures that satisfy the inequality \eqref{eq:FirstInequalityFormulation}, then
one has 
\begin{equation}\label{eq:InequalityCondition2}
 | \partial_t (\rho \star \mu) |  \leq 2\sqrt{\rho \star \nu} \sqrt{\rho \star \mu } \,,
 \end{equation}
with $\rho$ a smooth positive kernel defined on the whole domain $D$.
\end{lemma}

\begin{proof}
We will denote $\check{\rho}$ the adjoint of the convolution with $\rho$ for the $L^2$ scalar product. Then, inequality \eqref{eq:FirstInequalityFormulation} evaluated at $f = \check{\rho} \star g$ for $g \in \mathcal{D}$ gives 
\begin{equation*}
 \langle \rho \star \partial_t \mu , f \rangle^2 \leq 4 \langle \rho \star \nu , f \rangle  \langle \rho \star \mu ,  f \rangle \,.
\end{equation*}
Since $\partial_t (\rho \star \mu) = \rho \star \partial_t \mu$ and all the other terms $\rho \star \nu$ and $\rho \star \mu$ are smooth, inequality \eqref{eq:InequalityCondition2} is valid.
\end{proof}

\begin{definition}\label{def:FR}

Let $r:\R \times \R \mapsto \R_+ \cup \{+ \infty \}$ be the one-homogeneous convex function defined by
\begin{equation}
r(x,y) = 
\begin{cases}
\frac 14 \frac{y^2}{x} \text{ if } x > 0 \\
0 \text{ if } (x,y) = (0,0)\\
+\infty \text{ otherwise.}
\end{cases}
\end{equation}
The Fisher-Rao functional is defined on the product space of measures $\mathcal{M}^2(D)$ by
\begin{equation}\label{Eq:FisherRaoDefinition}
\on{FR}_f( \mu,\nu) \eqdef \int_D r\left(\frac{\ud \mu}{\ud\lambda}, \frac{\ud \nu}{\ud\lambda}\right) f \ud \lambda\,
\end{equation}
where $f \in C^0(D,\R_+^*)$ and $\lambda \in \mathcal{M}(D)$ that dominates $ \mu$ and $\nu$.
\end{definition}

\noindent
\textbf{Comments about the name Fisher-Rao:}\,
In the statistical literature, the Fisher-Rao functional has a slightly different meaning. Indeed, let us consider $\Theta : \R^n  \to \on{Prob}(M)$ be a map from a space of parameters to the space of smooth probability densities on a manifold $M$. On the space of smooth probability densities, one can use the metric $G(\rho)(\delta \rho,\delta \rho) = \int_M  \frac{(\delta \rho)^2}{\rho} \ud \mu$ where $\rho$ is the current density w.r.t. the chosen volume form $\mu$ and $\delta \rho$ is a tangent vector at $\rho$. The pull-back of this metric by $\Theta$ is called the Fisher-Rao metric. Thus, it is a metric on a space of parameters. Sometimes, the metric $G$ directly defined on the tangent space of probability densities is also called the Fisher-Rao metric, see \cite{Bauer:2014aa} for instance. In our case, we consider the same tensor on the space of all densities, that is \textit{without the constraint $\int \delta \rho \ud \mu = 0$}. Since the metric tensor has exactly the same formulation, we chose to keep the name Fisher-Rao metric.

We recall in the proof below the arguments to show that the Fisher-Rao functional is well defined. 
\begin{proof}
First, it is easy to check that $r$ is the Fenchel-Legendre conjugate of the indicator function $\iota_K$ of the convex set
\begin{equation*}
K \eqdef \left\{ (\xi_1,\xi_2) \in \R^2 : \xi_1 + \xi_2^2 \leq 0  \right\}\,.
\end{equation*}
We will write $\iota_K^*(x,y) = r(x,y)$ where $\iota_K^*$ is the Legendre conjugate of $\iota_K$.
We now consider the functional defined on $C_0(D)^2$ by 
\begin{equation}\label{Eq:DefinitionOfH}
H_f(u,v) \eqdef \int_D \iota_K(u/f,v/f) f \ud t \ud x\,
\end{equation}
with values in $\R_+ \cup \{+ \infty \}$. This functional is convex since $\iota_K$ is convex, lower semicontinuous and bounded below.
Now, the Fisher-Rao functional $\on{FR}_f(\mu)$ can be defined as the convex conjugate of $H$ on $\mathcal{M}^2(D)$.
We have
\begin{equation*}
H_f^*(\mu,\nu) = \sup_{u,v \in C^0(D)} \left[ \int_D u \ud \mu + \int_D v \ud \nu - \int_D \iota_K(u/f,v/f) f \ud t \ud x \right]  \,.
\end{equation*}
Using \cite[Theorem 5, page 457]{rockafellar1971integrals}, the following formula holds
\begin{equation*}
H_f^*(\mu,\nu) = \int_D r\left( \frac{\ud \mu}{\ud \lambda_0},\frac{\ud \nu}{\ud \lambda_0}\right)f \ud \lambda_0  +  \int_D r^\infty\left(\frac{\ud \mu^\perp}{\ud \lambda_0},\frac{\ud \nu^\perp}{\ud \lambda_0}\right) f \ud \sigma \,,
\end{equation*}
where $\lambda_0 \eqdef \ud x \ud t$ is the Lebesgue measure on $D$ and $\sigma$ is any measure that dominates the singular parts of $\mu, \nu$ w.r.t. $\lambda_0$, which are denoted by $\mu^\perp$ and $\nu^\perp$.
Here $r^\infty$ denotes the recession function of $r$ and since $r$ is one homogeneous, one has $r^\infty = r$.

Now, by one homogeneity, one can simply use a measure $\lambda$ that dominates $\mu$ and $\nu$ and use the formula
\begin{equation*}
\on{FR}_f(\mu,\nu) = \int_D r \left(\frac{\ud \mu}{\ud \lambda}, \frac{\ud \nu}{\ud\lambda}\right) f \ud \lambda\,.
\end{equation*}
\end{proof}

\begin{proposition}\label{Th:SubDifOfFR}
The subdifferential of the Fisher-Rao functional $\on{FR}_f$ at a point $(\mu,\nu)$ in its domain satisfies
\begin{equation*}
\left\{ (u,v) \in C^0(D)^2 \, : \, H_f(u,v) = 0 \text{ and } \langle u, \mu \rangle +  \langle v, \nu \rangle = \on{FR}_f(\mu,\nu) \right\} \subset \partial \on{FR}_f(\mu,\nu)\,.
\end{equation*}
\end{proposition}

\begin{proof}
The condition $(u,v) \in \partial \on{FR}_f(\mu,\nu)$ is known to be equivalent to 
\begin{equation*}
\on{FR}^*_{f}(u,v) + \on{FR}_{f}(\mu,\nu) = \langle \mu,u\rangle + \langle \nu,v\rangle\,.
\end{equation*}
Since $(\mu,\nu) \in \on{Dom}(\on{FR}_f)$, $\on{FR}^*_{f}(u,v)$ has to be finite. By one-homogeneity of $\on{FR}_f$, its Fenchel-Legendre conjugate takes its values in $\{ 0,+\infty \}$ and thus $\on{FR}_{f}^*(u,v)=0$. Therefore, Young's inequality is equivalent to $\on{FR}^*_{f}(u,v)=0$ and $\on{FR}_{f}(\mu,\nu) = \langle \mu,u\rangle + \langle \nu,v\rangle$.
Moreover, we know that $\on{FR}^*_f = H_f$ since $H_f$ is a convex and l.s.c function bounded below by $0$. Therefore, we obtain $H_{f}(u,v) = 0$ and the desired condition.
\end{proof}

\begin{notation}
Let $\mu,\nu,\sigma$ be measures in $\mathcal{M}(D)$. We denote 
\begin{equation}\label{eq:FisherRaoInequalityFormulation}
\on{FR}(\mu,\nu) \preccurlyeq \sigma \,
\end{equation}
if for all $f \in C^0(D,\R_+^*)$ one has,
\begin{equation*}
\on{FR}_f(\mu,\nu) \leq  \langle \sigma , f \rangle \,.
\end{equation*}
\end{notation}

\begin{proposition}\label{Th:ShortFormulation}
Let $\mu,\nu \in \mathcal{M}(D)$ be two Radon measures. Inequalities \eqref{eq:FirstInequalityFormulation} and \eqref{eq:FisherRaoInequalityFormulation} are equivalent. More precisely, we have $$\partial_t \mu \text{ is a measure and } \on{FR}(\mu,\partial_t\mu) \preccurlyeq \nu \, \Leftrightarrow \, \langle \mu , \partial_t f \rangle^2 \leq 4 \langle \nu , f \rangle  \langle \mu ,  f \rangle \text{ for all }f \in   C^0(D,\R_+^*)\,.$$
\end{proposition}

\begin{proof}
First remark that the equivalence is true if $\mu,\nu$ are smooth functions. We present a proof based on regularization arguments. 
Assume first that the inequality \eqref{eq:FirstInequalityFormulation} is satisfied, then using Lemma \ref{th:RegularisationOfFirstInequality}, one has
\begin{equation*}
 | \partial_t (\rho \star \mu) |  \leq 2\sqrt{\rho \star \nu} \sqrt{\rho \star \mu } \,,
 \end{equation*}
 for a kernel on the domain $D$. It implies 
 \begin{equation*}
 | \partial_t \sqrt{\rho \star \mu} |^2 \leq \rho \star \nu\,,
 \end{equation*}
 and therefore, for every $f\in C^0(D,\R_+^*)$,
  \begin{equation*}
\on{FR}_f(\rho \star \mu,\partial_t (\rho \star \mu)) \leq  \int_D \rho \star \nu \, f \ud x \ud t\,.
 \end{equation*}
 Let us introduce a sequence $\rho_n$ converging to the Dirac measure. The weak lower semicontinuity of the Fisher-Rao functional leads to
 \begin{equation*}
\on{FR}_f(\mu,\partial_t\mu) \leq \liminf_{n \to \infty} \on{FR}_f(\rho_n \star \mu,\partial_t (\rho_n \star \mu)) \leq \int_D  \nu \, f \ud x \ud t\,,
 \end{equation*}
 which is the Fisher-Rao inequality \eqref{eq:FisherRaoInequalityFormulation}.
 
 Now, let us assume that the Fisher-Rao inequality \eqref{eq:FisherRaoInequalityFormulation} is satisfied. Note first that it is sufficient to prove the inequality in Lemma \eqref{th:RegularisationOfFirstInequality} for test functions that satisfy $\int_D f \ud x \ud t = 1$ since the inequality is one homogeneous with respect to $f$. Using the Jensen inequality on the Fisher-Rao functional, we obtain
 \begin{equation}\label{eq:IneqIntermediaire43}
 \on{FR}_f(\rho \star \mu,\partial_t (\rho \star \mu)) \leq  \on{FR}_{\rho \star f}(\mu,\partial_t \mu) \leq \langle \rho \star \nu,f \rangle\,.
 \end{equation}
 Then, using the Cauchy-Schwarz inequality, one has
 \begin{equation*}
 | \langle f, \partial_t (\rho \star \mu) \rangle| = 2 | \langle f, (\partial_t \sqrt{\rho \star \mu}) \rho \star \mu \rangle| \leq 2\on{FR}_f(\rho \star \mu,\partial_t (\rho \star \mu))^{1/2} \langle f , \rho \star \mu \rangle^{1/2}\,,
 \end{equation*}
 which implies, together with \eqref{eq:IneqIntermediaire43},
  \begin{equation*}
 | \langle f, \partial_t (\rho \star \mu) \rangle| \leq 2\langle\rho \star \nu,f \rangle^{1/2} \langle f , \rho \star \mu \rangle^{1/2}\,.
 \end{equation*}
 Applying this inequality with $\rho = \rho_n$ and taking the limit, one gets inequality \eqref{th:RegularisationOfFirstInequality}.
 \end{proof}

\begin{notation}
Due to Proposition \ref{Th:ShortFormulation}, we now write $\on{FR}(\mu,\partial_t\mu) \preccurlyeq \nu  $ in replacement of the condition $\langle \mu , \partial_t f \rangle^2 \leq 4 \langle \nu , f \rangle  \langle \mu ,  f \rangle$, therefore omitting to precise that $\partial_t \mu$ is in $\mathcal{M}(D)$.
\end{notation}

We end up this section with an important regularity result that will be necessary for the formulation of the first-order optimality condition. 

\begin{definition}
Let $\alpha \in (0,1)$. 
The space $C^{0,\alpha}([0,1],\mathcal{M})$ is the space of H\"older continuous paths of measures endowed with the bounded Lipschitz distance. It is defined by 
\begin{multline}
C^{0,\alpha}([0,1],\mathcal{M}) = \{ \mu \in L^1([0,1],\mathcal{M}) \, : \, \exists \, M>0 \,\text{ s.t. }  \, \| \mu(t)-\mu(s)\|_{BL} \leq M |t-s|^{\alpha} \}\,,
\end{multline}
where $\| \mu(t) - \mu(s) \|_{BL} \eqdef \sup \, \{\langle f , \mu(t) \rangle - \langle f , \mu(s) \rangle \, ; \, \| f\|_{\infty} \leq 1 \text{ and } \on{Lip}(f) \leq 1 \}$.
\end{definition}
The bounded Lipschitz distance metrizes the weak convergence on bounded sets of the space of Radon measures. The result hereafter proves that we will deal with paths of measures on the space $[0,1]$ instead than measures on $D$. This distance is weaker than the dual norm.

\begin{proposition}\label{Th:TimeRegularity}
Let $\mu,\partial_t \mu$ be Radon measures on $D$ such that $\on{FR}_f(\mu,\partial_t \mu) < \infty$ for some function $f \in C^0(D,\R_+^*)$. Then, $\mu \in C^{0,1/2}([0,1],\mathcal{M}_+([0,1]))$. 
\end{proposition}

\begin{proof}
Since $f\in C^0(D,\R_+^*)$ is bounded below by a positive constant, it is sufficient to prove the result in the case where $f\equiv 1$.
We present a proof by regularization. Let us first study the case where $\mu \in C^0([0,1],L^1([0,1]))$ and in addition $\mu(t) \geq 0$ for all $t\in [0,1]$. We first note that

\begin{align}
\| \sqrt{\mu}(s) \|_{L^2} &\leq \| \sqrt{\mu}(0) \|_{L^2} + \int_0^s \| \partial_t \mu \|_{L^2} \ud t \nonumber\\ 
& \leq \mu(0)([0,1]) + \sqrt{s} \sqrt{ \on{FR}_f(\mu,\partial_t \mu)}\label{Eq:IneqBoundedL2}
\end{align}
by application of the Cauchy-Schwarz inequality. This shows that $\sqrt{\mu}$ is bounded in $L^2$. It also implies
\begin{equation}\label{Eq:GlobalIneq}
\mu(0)([0,1])  \leq \mu(D) + \sqrt{ \on{FR}_f(\mu,\partial_t \mu)}\,.
\end{equation}
\\
We have, by standard estimations, recalling that $\mu \geq 0$,
\begin{align*}
|\langle \mu(t) - \mu(s) , \psi \rangle | & \leq \| \mu(t) - \mu(s)  \|_{L^1}\| \psi \|_\infty\\
& \leq \langle |\sqrt{\mu(t)} -\sqrt{\mu(s)}|,  |\sqrt{\mu(t)} +\sqrt{\mu(s)|} \rangle \| \psi \|_\infty\\
& \leq \| \sqrt{\mu(t)} +\sqrt{\mu(s)} \|_{L^2} \left\| \int_s^t \partial_t \mu(\tau) \ud \tau \right\|_{L^2}  \| \psi \|_\infty\\
& \leq 2 \left(\sup_{\tau \in [0,1]} \| \sqrt{\mu(\tau)} \|_{L^2} \right)\int_s^t \| \partial_t \mu(\tau) \|_{L^2} \ud \tau   \| \psi \|_\infty\\
& \leq 2 \left(\sup_{\tau \in [0,1]} \| \sqrt{\mu(\tau)} \|_{L^2} \right) \sqrt{|t-s|} \sqrt{ \on{FR}_f(\mu,\partial_t \mu)}   \| \psi \|_\infty\,,
\end{align*}
where the last inequality again comes from the application of the Cauchy-Schwarz inequality.
\\
We now prove the result by density with $\rho_{n}$ a convolution kernel on the domain $D$ converging to the Dirac measure. 
The sequence $\rho_{n} \star \mu$ belongs to $C^{0,1/2}([0,1],\mathcal{M}([0,1]))$ with a H\"older  constant bounded by $\sup_{n \in \N}  \on{FR}_f(\rho_{n} \star \mu,\partial_t (\rho_{n} \star \mu)) $. Since, for $f \equiv 1$ we have $\rho_n \star f = f$, we have $$\on{FR}_f(\rho_{n} \star \mu,\partial_t (\rho_{n} \star \mu)) \leq \on{FR}_{\rho_{n} \star f} ( \mu,\partial_t \mu)) =  \on{FR}_{f}( \mu,\partial_t \mu))\,,$$
which implies that the H\"older constant is bounded uniformly. Therefore, by the Arzela-Ascoli theorem which can be applied here since bounded sets for the dual norm are compact in $(\mathcal{M}_+([0,1]), \| \cdot \|_{BL})$, the limit is also in $C^{0,1/2}([0,1],\mathcal{M}_+([0,1]))$.

We now give an explicit estimation of the Hölder constant. We have
\begin{multline*}
| \langle \mu(t) - \mu(s) , \psi \rangle |  \leq \| \rho_n \star \mu(t) -  \rho_n \star \mu(s)  \|_{L^1}\| \psi \|_\infty \\+ | \langle \rho_n \star \mu(s) -\mu(s) ,\psi \rangle | + | \langle \rho_n \star \mu(t) -\mu(t) ,\psi \rangle |\,.
\end{multline*}
The two last terms can be made arbitrarily small with $n \to \infty$ and the first term can be bounded using the previous case since $\rho_n \star \mu \in C^0([0,1],L^1([0,1]))$. We thus apply the previous inequalities to get
\begin{align*}
| \langle \mu(t) - \mu(s) , \psi \rangle |  \leq 2  \left(\lim_{n \to \infty}\sup_{\tau \in [0,1]} \| \sqrt{\rho_n \star\mu(\tau)} \|_{L^2} \right) \sqrt{|t-s|} \sqrt{ \on{FR}_f(\mu,\partial_t \mu)} \| \psi \|_\infty \,.
\end{align*}
Last, we have using inequality \eqref{Eq:IneqBoundedL2}, a bound on the first term of the r.h.s.
\begin{equation*}
\lim_{n \to \infty}\sup_{\tau \in [0,1]} \| \sqrt{\rho_n \star\mu(\tau)} \|_{L^2}   \leq \mu(D) + \sqrt{ \on{FR}_f(\mu,\partial_t \mu)}\,.
\end{equation*}
Indeed, we have that $\lim_{n \to \infty} \rho_{n} \star \mu (D) =  \mu(D)$. 
\end{proof}

\begin{remark}
We have actually proved that the path $\mu(t)$ is Hölder with respect to the dual norm on the space of measures with a constant which is explicit in terms of $\mu(0)([0,1])$ and $\on{FR}_f(\mu,\partial_t \mu)$.
\end{remark}
\section{Minimization of the acceleration}\label{Sec:Relaxation}

In this section, we are interested in the minimization of the acceleration on $(Q,G)$ which represents the group of diffeomorphisms $\on{Diff}_0^2([0,1])$. We first need that the acceleration functional is coercive which is the consequence of a general result on Hilbert manifolds and of the right-invariance of the metric. Then, we study this variational problem in the particular case of $(Q,G)$.

\subsection{A general lemma and its corollaries on the group of diffeomorphisms}

\begin{notation}
Let $(M,g)$ be a Hilbert manifold and $x(t) \in H^2([0,T],M)$.
We will denote by
$\frac{D}{Dt}\dot{x}$ the covariant derivative of $\dot{x}$ along $x(t)$.
\end{notation}

\begin{lemma}\label{th:CSonAcceleration}
Let $(M,g)$ be a Hilbert manifold and $x(t) \in H^2([0,T],M)$ such that $\dot{x}(0)=0$. Then, one has
\begin{equation}\label{eq:CauchySchwarzAcceleration}
\int_0^T g(\dot{x}, \dot{x}) \ud s \leq 4T^2 \int_0^T g\left(\frac{D}{Dt}\dot{x}, \frac{D}{Dt}\dot{x}\right) \ud s\,.
\end{equation}
\end{lemma}

\begin{proof}
We start with $\frac 12 \frac{\ud }{\ud t} g(\dot{x},\dot{x}) = g\left(\frac{D}{Dt}\dot{x}, \dot{x}\right) $ by definition of the covariant derivative.
By applications of the Cauchy-Schwarz inequality on the previous formula and time integration, we have
\begin{equation*}
\frac 12 \frac{\ud }{\ud t} g(\dot{x},\dot{x}) \leq \sqrt{g(\dot{x},\dot{x})}\sqrt{g\left(\frac{D}{Dt}\dot{x},\frac{D}{Dt}\dot{x}\right)}\,,
\end{equation*}
which implies that for all $s \leq t$
\begin{align*} 
&\frac 12 g(\dot{x},\dot{x})(s) \leq \int_0^t \sqrt{g(\dot{x},\dot{x})}\sqrt{g\left(\frac{D}{Dt}\dot{x},\frac{D}{Dt}\dot{x}\right)} \ud s \,,\\
&\frac 12 g(\dot{x},\dot{x})(s) \leq \sqrt{\int_0^t g(\dot{x},\dot{x}) \ud s} \sqrt{\int_0^t g\left(\frac{D}{Dt}\dot{x},\frac{D}{Dt}\dot{x}\right) \ud s}\,.
\end{align*}
By integration, we get
\begin{equation}
\frac 12 \int_0^T g\left(\dot{x},\dot{x}\right) \ud s \leq T \sqrt{\int_0^T g(\dot{x},\dot{x}) \ud s} \sqrt{\int_0^T g\left(\frac{D}{Dt}\dot{x},\frac{D}{Dt}\dot{x}\right) \ud s}
\end{equation}
which gives the result.
\end{proof}

We now use this result on the group of diffeomorphisms $\on{Diff}^s(N)$ where $N$ is a compact manifold and the metric on the group is a right-invariant Sobolev metric of the same order $s$.
\begin{definition}\label{Def:AccelerationFunctional}
Let $(M,g)$ be a Hilbert manifold and $\mathcal{C}$ be the set of paths defined by
\begin{equation*} 
\mathcal{C}^{0,1} \, = \{ x(t) \in \! H^2([0,T],M) :  (x(0),\dot{x}(0)) = (x_0,v_0) \in TM \! \text{ and } (x(1),\dot{x}(1)) = (x_1,v_1) \}\,.
\end{equation*}
The acceleration functional $\mathcal{J}$ is defined by 
\begin{equation}\label{eq:AccelerationFunctional}
\mathcal{J}(x) \eqdef \int_0^T g\left(\frac{D}{Dt}\dot{x}, \frac{D}{Dt}\dot{x}\right) \ud s\,,
\end{equation}
subject to the boundary constraints $(x(0),\dot{x}(0)) = (x_0,v_0) \in TM$ and $(x(1),\dot{x}(1)) = (x_1,v_1) \in TM$.
\par 
We also define the set of unconstrained path at time $1$  
\begin{equation*} 
\mathcal{C}^{0} \, = \{ x(t) \in \! H^2([0,T],M) :  (x(0),\dot{x}(0)) = (x_0,v_0) \in TM \}\,.
\end{equation*}
\end{definition}

\begin{theorem}\label{th:AccelerationCoercive}
Let $(N,g)$ be a compact manifold.
On $\on{Diff}^s(N)$, the acceleration functional \eqref{eq:AccelerationFunctional} is coercive on $\mathcal{C}$ endowed with the topology of $H^2([0,1],H^s(N))$.
\end{theorem}
\begin{proof}
The geodesic energy $\int_0^T g(\dot{x}, \dot{x})$ is bounded above by the functional \eqref{eq:AccelerationFunctional}. 
We now use the fact that the topology on the $\on{Diff}^s(N)$ is stronger than that of $H^2([0,1],H^s(N))$ (see \cite{BruverisVialard}). In coordinates, the acceleration can be written as 
\begin{equation*}
\frac{D}{Dt}\dot{x}= \ddot{x} + \Gamma(x)(\dot{x},\dot{x})\,,
\end{equation*}
where $\Gamma$ denotes the Christoffel symbols associated with the right-invariant metric (see for instance \cite[Section 3.2]{HOSplines2} for the finite dimensional case of a Lie group with right-invariant metric). By smoothness of the metric, the Christoffel symbols are thus bounded on a neighborhood of identity. Then, by right invariance of the metric, the Christoffel symbols is a bounded bilinear operator on every metrically bounded ball. 
Therefore, $x$ is bounded in $H^2([0,1],H^s(N))$.

\end{proof}

\begin{corollary}
Let $(N,g)$ be a compact manifold.
On $\on{Diff}^s(N)$, any minimizing sequence for the acceleration functional \eqref{eq:AccelerationFunctional} is
\begin{enumerate}
\item bounded in $H^2([0,1],H^s(N))$,
\item bounded in $C^1([0,1],\on{Diff}^s(N))$.
\end{enumerate}
\end{corollary}

\begin{proof}
The first point is Theorem \eqref{th:AccelerationCoercive}. 
Let us prove the second point. Using Lemma \ref{th:CSonAcceleration}, one can apply \cite[Theorem 3.1]{BruverisVialard} which gives that the path is bounded in $C^0([0,1],\on{Diff}^s(N))$. 
Recall that on every metrically bounded ball the topologies of $\on{Diff}^s(N)$ and $H^s(N)$ are equivalent.
 Since the path is also bounded in $H^2([0,1],H^s(N))$, it gives that the path is bounded in $C^1([0,1],\on{Diff}^s(N))$. 
\end{proof}

In order to study the relaxation problem, we will need a controllability lemma which is valid on an infinite dimensional Riemannian manifold. 

\begin{lemma}\label{Th:LocalCubics}
Let $(M,g)$ be a Riemannian manifold possibly of infinite dimensions and a $C^1$ curve $c(t) \in M$ such that $c(0) = x \in M$. Let $V \subset T_xM$ be an open neighborhood of $0$ on which the exponential map is a diffeomorphism and we denote by $\log_x$ the inverse of this map. Consider the map, for any $t>0$,
\begin{align}
R_c(t): (T_xM)^2 &\to TM\\
 (u_1,u_2) & \mapsto (\exp_x(z(t,u_1,u_2)), d\!\exp_x(z(t,u_1,u_2))(\dot{c}(0) + t u_1 + \frac 12 t^2 u_2))\,,
\end{align}
where $z(t,u_1,u_2) = \log_x(c(t)) + \frac 12 t^2 u_1 + \frac 16 t^3 u_2$ and $\partial_t z(t,u_1,u_2) = \frac{\ud}{\ud t} \log_x(c(t)) + t u_1 + \frac 12 t^2 u_2$.
For $t>0$ and small enough, the map $R_c$ is a local diffeomorphism at $(0,0) \in (T_xM)^2$. 
\end{lemma}
\begin{proof}
We first treat the case when $M$ is the Euclidean space denoted by $E$. In this case, the curve $z$ is $z(t,u_1,u_2) = c(t) + \frac 12 t^2 u_1 + \frac 16 t^3 u_2$. The differential of the map $R_c(t): E^2 \mapsto E$ is given by 
\begin{equation}
d R_c(t) (0,0) = \begin{pmatrix} (t^2/2) \on{Id} & (t^3/6) \on{Id}\\ t \on{Id} &(t^2/2) \on{Id} \end{pmatrix}
\end{equation}
which is invertible for all time $t>0$.
For $t,u_1,u_2$ small enough, $z(t,u_1,u_2)$ lies in the neighborhood $V$. 
\par 
Let us treat the general case of a manifold. Since the exponential map is a local diffeomorphism on $V$, the map $d\!\exp: TT_xM \mapsto TM$ is also a local diffeomorphism for any element $(v,w)\in TT_xM \simeq (T_xM)^2$ such that $v \in V$. Therefore, it is sufficient to prove the result on the map
$S: (T_xM)^2 \to TT_xM$ defined by $S(t,u_1,u_2) = (z(t,u_1,u_2),z'(t,u_1,u_2))$ which reduces to the Euclidean case treated above.
\end{proof}
This lemma will be used to extend the relaxation result to general initial conditions on the path and general endpoint conditions on the defect measures by developping a perturbation argument.
\begin{remark}\label{Rem:Norm}
When the curve $c$ lies in $H^2([0,1],M)$, a direct estimation leads to 
\begin{equation}
\|R_c(\cdot,u_1,u_2) - c \|_{H^2([0,1],M)} \leq k \max(\| u_1\|,\| u_2\|)
\end{equation}
where $R_c(\cdot,u_1,u_2)$ denotes the path defined in Lemma \ref{Th:LocalCubics}. The constant $k$ is local and depends on the metric in the neighborhood $V$.
\end{remark}

\subsection{The case of $\on{Diff}_0^2([0,1])$}
\begin{notation}
Recall the geodesic equation
\begin{equation}\begin{cases}
\dot{q} = \eta(q)p\\
\dot{p} = - \frac 12 \int_x^1 \eta(q)p^2 \ud y +  \begin{bmatrix} 1 & \varphi \end{bmatrix} H_2^{-1}  \begin{bmatrix} a \\ b +c \end{bmatrix}
\end{cases}
\end{equation}
with $a,b,c$ the coefficients defined in \eqref{eq:FirstProjection}, \eqref{eq:SecondProjection1} and \eqref{eq:SecondProjection2} and define
\begin{equation} \label{Eq:NonLocalOperator}
\mathcal{U}(f,q)(x) \eqdef  \frac 12 \int_x^1 \eta(q) f \ud y\,.
\end{equation}
The projection $\begin{bmatrix} 1 & \varphi \end{bmatrix} H_2^{-1}  \begin{bmatrix} a \\ b +c  \end{bmatrix} $ will be decomposed into two terms:

\begin{align*}
&\pi_1(p,q) \eqdef \begin{bmatrix} 1 & \varphi \end{bmatrix} H_2^{-1}  \begin{bmatrix} 0 \\ b  \end{bmatrix} \\
&\pi_2(q,f) \eqdef (\on{id} - \pi^*_q)(\mathcal{U}(f,q))\,.
\end{align*}
We will often omit arguments in $\pi_1$, $\pi_2$ and $\mathcal{U}$ when there is no possible confusion.
\end{notation}
The decomposition into two parts for the projection is used for clarity in the next proof. Indeed, $\pi_1$ is continuous w.r.t. the weak topology whereas $\pi_2$ is not.

\begin{definition}[Minimization of the acceleration]
On the Hilbert manifold $Q$, the minimization of the acceleration can be rewritten as the minimization of
\begin{equation}\label{Eq:InitialAccelerationFunctionalWithBoundary}
\mathcal{J}_0(p,q) = \int_{D} \! \eta(q)  \left( \dot{p}  + \mathcal{U}(p^2,q) -  \pi_1 - \pi_2 \right)^2 \! \ud x \ud t\,,
\end{equation}
on the set of curves 
\begin{multline} 
\mathcal{C}^{0,1} \eqdef \Big\{ (p(t),q(t)) \in T^*Q \, : \, q \in H^2([0,1],L^2([0,1])) \text{ such that } \dot{q} = \eta(q)p \\ \text{ and under the  boundary constraints } q(0) = q_0\,, \,q(1)=q_1  \text{ and } p(0) = p_0 \in T_{q_0}^*Q\,, p(1)=p_1 \in T_{q_1}^*Q\Big\}\,.
\end{multline} 
We shall be mainly interested in similar minimization problems but for a relaxed endpoint constraint, namely
\begin{equation}\label{eq:DefinitionRelaxedFunctional}
\mathcal{J}(p,q)=
\mathcal{J}_0(p,q) + \frac{1}{\sigma_1^2} \| p(1)-p_1 \|^2_{L^2} + \frac{1}{\sigma_2^2}  \| q(1)-q_1 \|^2_{L^2}\,,
\end{equation}
where $\sigma_1, \sigma_2$ are two positive constants. The set of admissible curves will be denoted by $\mathcal{C}^{0}$ is defined as above but without the constraint at time $1$, namely $q(1) = q_1$ and $p(1) = p_1$. 
\end{definition}

Due to the square on $p$ in $\mathcal{U}$, the functional $\mathcal{J}$ is expected to fail being lower semi continuous with respect to the weak topology on $\mathcal{C}$. We now introduce the set of relaxed curves.

\begin{definition}[Set of relaxed curves]
The set of relaxed curves is 
\begin{multline} 
\mathcal{C}_R^{0} \eqdef \Big\{ (p,q,\mu, \nu) \in \mathcal{C}^0 \times \mathcal{M}^2 \, : \, \mu \geq p^2 \,, \nu \geq \dot{p}^2 \, \text{ and } \on{FR}(\mu - p^2,\partial_t(\mu - p^2))  \preccurlyeq \nu-\dot{p}^2 \\ \text{under the  boundary constraints } q(0) = q_0\,, \, p(0) = p_0 \in T_{q_0}^*Q \text{ and } \mu(0) = p(0)^2 \Big\}\,.
\end{multline}
We also denote by $\mathcal{C}_R^{0,1}$ the set as above under the additional constraint at time $1$ given by 
$q(1) = q_1$, $p(1) = p_1 \in T_{q_1}^*Q$ and $\mu(1) = p(1)^2$.\\
In the sequel, the notation $\mathcal{C}_R$ stands for either $\mathcal{C}_R^{0}$ or $\mathcal{C}_R^{0,1}$ and we will use the same convention for $\mathcal{C}$.
\end{definition}
\ifcomments
\begin{remark}[Open question]\label{Re:OpenQuestion1}
We do not know if the relaxation of $\mathcal{C}^{0,1}$ is actually $\mathcal{C}^{0,1}_R$ or only a subset of it.
\end{remark}
\fi
\begin{proposition}
The sets $\mathcal{C}_R^{0}$ and $\mathcal{C}_R^{0,1}$ are closed in the weak topology on $\mathcal{C} \times \mathcal{M}^2(D)$.
\end{proposition}
\begin{proof}
Let $(p_n,q_n,\mu_n,\nu_n) \in \mathcal{C}_R$ be a weakly convergent sequence in $\mathcal{C} \times \mathcal{M}^2$ to $(p,q,\mu, \nu) \in \mathcal{C}_R$.
Up to extracting subsequences, we can assume in addition that $p_n^2 \rightharpoonup \pi$, $\dot{p}_n^2 \rightharpoonup \delta$ in $\mathcal{M}$. By passing to the limit and using Fatou's lemma on $\mu_n \geq p_n^2$ and $\nu_n\geq \dot{p}^2 $, we first obtain $\mu \geq \pi \geq p^2$ and $\nu \geq \delta \geq \dot{p}^2$ which are the two first conditions on the measures $\mu,\nu$.

On one hand, we first write $p_n = p + \alpha_n$ with $\alpha_n \rightharpoonup 0$, $\alpha_n^2 \rightharpoonup \pi - p^2$ and $\dot{\alpha_n}^2 \rightharpoonup \delta - \dot{p}^2$. By Cauchy-Schwarz inequality, we have for every $f \in C^\infty([0,1],\R_+^*)$,
\begin{equation*}
\on{FR}_f(\alpha_n^2,\partial_t (\alpha_n^2)) \leq \langle \dot{\alpha_n}^2, f \rangle\,.
\end{equation*}
The right-hand side is converging to $\langle \delta - \dot{p}^2, f \rangle$. Since the Fisher-Rao functional is lower semi-continuous, passing to the limit gives
\begin{equation}\label{eq:IneqFR1}
\on{FR}_f(\pi - p^2,\partial_t (\pi - p^2)) \leq \langle \delta - \dot{p}^2, f \rangle\,.
\end{equation}
Moreover, by assumption, we have the inequality 
\begin{equation*}
\on{FR}_f(\mu_n - p_n^2,\partial_t(\mu_n - p_n^2)) \leq \langle \nu_n - \dot{p}_n^2, f \rangle\,,
\end{equation*}
on which, the same previous arguments give
\begin{equation}\label{eq:IneqFR2}
\on{FR}_f(\mu -\pi,\partial_t (\mu - \pi)) \leq \langle \nu - \delta, f \rangle\,.
\end{equation}
On the other hand, the Fisher-Rao functional is subadditive (one-homogeneous and convex) which implies
\begin{equation*}
\on{FR}_f(\mu - p^2,\partial_t(\mu - p^2)) \leq \on{FR}_f(\mu - \pi,\partial_t(\mu-\pi))  + \on{FR}_f(\pi - p^2,\partial_t(\pi - p^2))\,,
\end{equation*}
which, using inequality \eqref{eq:IneqFR2}, gives
\begin{equation*}
\on{FR}_f(\mu - p^2,\partial_t(\mu - p^2)) \leq \langle \nu - \delta, f \rangle  + \on{FR}_f(\pi - p^2,\partial_t(\pi-p^2))\,.
\end{equation*}
Introducing $\dot{p}^2$ in the dual pairing, this last inequality can be rewritten as
\begin{equation*}
\on{FR}_f(\mu - p^2,\partial_t(\mu - p^2)) \leq \langle \nu - \dot{p}^2, f \rangle + \langle \dot{p}^2 - \delta, f \rangle +  \on{FR}_f(\pi - p^2,\partial_t(\pi - p^2))  \,.
\end{equation*}
By inequality \eqref{eq:IneqFR1}, we have 
$\langle \dot{p}^2 - \delta, f \rangle +  \on{FR}_f(\pi - p^2,\partial_t(\pi - p^2)) \leq 0$.
We thus get the result
\begin{equation*}
\on{FR}_f(\mu - p^2,\partial_t(\mu - p^2)) \leq \langle \nu - \dot{p}^2, f \rangle  \,.
\end{equation*}

The last condition to be checked is that $(p,q) \in T^*Q$, which is also true since the constraints defining $Q$ are weakly continuous: $\lim_{n\to \infty} \int_0^1 q_n \ud x = \int_0^1 q \ud x$ and $\lim_{n\to \infty} \int_0^1 \eta(q_n) \ud x = \int_0^1 \eta(q) \ud x$ and Lemma \ref{th:WeakConvergenceLemma} gives that $\pi_{q_n}(p_n) = p_n$ weakly converges to $\pi_q(p)$ and therefore $\pi_q(p) = p$.
\par
Last, the boundary constraints are trivially satisfied under weak convergence.
\end{proof}

\begin{remark}
We could have written the relaxed set $\mathcal{C}_R$ in terms of the defect measures $\mu - p^2$ and $\nu -\dot{p}^2$. However, the corresponding relaxed acceleration functional would have been only lower semi-continuous and not continuous with respect to these defect measures. Although it does not change the result, we prefer working with a continuous relaxed functional, as defined below.
\end{remark}

\begin{definition}[Relaxed acceleration functional]
The relaxed functional defined on $\mathcal{C}_{R}^{0,1}$ is  
\begin{multline}\label{eq:RelaxedAccelerationFunctional}
\mathcal{J}_R^{0,1}(p,q,\mu,\nu)\! = \! \int_D\! \eta \ud \nu \!+ \!\int_{D} \! \eta \Big( [\dot{p}  + \mathcal{U}(\mu,q) -  \pi_1(p,q) - \pi_2(\mu,q) ]^2 - \dot{p}^2\Big) \! \ud x \!\ud t \,.
\end{multline}
The relaxed acceleration functional for a soft constraint at time $1$ defined on $\mathcal{C}_R^{0}$ is
\begin{multline}\label{eq:RelaxedAccelerationFunctional2}
\mathcal{J}_R^0(p,q,\mu,\nu)\! = \! \int_D\! \eta \ud \nu \!+ \!\int_{D} \! \eta \Big( [\dot{p}  + \mathcal{U}(\mu,q) -  \pi_1(p,q) - \pi_2(\mu,q) ]^2 - \dot{p}^2\Big) \! \ud x \!\ud t + P(p(1),q(1))\,,
\end{multline}
where $P$ is a functional on $T^*Q$ which is the penalization term.
The notation $\mathcal{J}_R$ stands for either $\mathcal{J}_R^{0,1}$ or $\mathcal{J}_R^{0}$.
\end{definition}

\begin{proposition}\label{th:Continuity}
The functional $\mathcal{J}_R^{0,1}$ is lower semi-continuous with respect to the weak topology. 
\par If $P$ is also l.s.c with respect to the weak topology on $(L^2[0,1])^2$ it is also the case for $\mathcal{J}_R^{0}$.
\end{proposition}

\begin{proof}
Let $(p_n,q_n,\mu_n,\nu_n) \in \mathcal{C}_R$ be a weakly convergent sequence in $\mathcal{C} \times \mathcal{M}^2$.
The first term in \eqref{eq:RelaxedAccelerationFunctional} is the dual pairing between a weakly convergent sequence of measures and a strongly convergent sequence in $C^0(D)$, which gives the continuity of the first term.

In order to prove the continuity of the second term, we expand it into:
\begin{multline}\label{eq:Expand}
\int_{D}  2 \, \eta(q) \, \dot{p} \, (\mathcal{U}(\mu,q) -  \pi_1(p,q) - \pi_2(\mu,q))  + \eta(q) \Big(\mathcal{U}(\mu,q) -  \pi_1(p,q) - \pi_2(\mu,q)\Big)^2  \! \ud x \ud t\,.
\end{multline}
Using Lemma \ref{th:WeakConvergenceLemma}, we have that $\mathcal{U}(\mu_n,q_n) -  \pi_1(p_n,q_n) - \pi_2(\mu_n,q_n)$ strongly converges in $C^0(D)$. Therefore the second term in \eqref{eq:Expand} is strongly convergent in $C^0(D)$. The first term is the $L^2$ pairing between a weakly convergent sequence $\dot{p}_n$ and a strongly convergent sequence $\eta(q_n)\, (\mathcal{U}(\mu_n,q_n) -  \pi_1(p_n,q_n) - \pi_2(\mu_n,q_n))$. 
\end{proof}

\begin{remark}
The penalization terms $\| p(1)-p_1 \|^2_{L^2}$  and $ \| q(1)-q_1 \|^2_{L^2}$ are lower semi-continuous and the previous proposition applies to the functional \eqref{eq:DefinitionRelaxedFunctional}.
\end{remark}

\begin{theorem}\label{Th:Relaxation}
If $P$ is assumed continuous w.r.t. the weak topology then 
the continuous functional \eqref{eq:RelaxedAccelerationFunctional2} defined on $\mathcal{C}_{R}^{0}$ is the relaxation of the acceleration functional. Namely, for every $(p,q,\mu,\nu) \in \mathcal{C}_R^0$, there exists a sequence $(p_n,q_n) \in \mathcal{C}^0$ such that 
\begin{equation}\limsup_{n\to \infty} \mathcal{J}(p_n,q_n) \leq \mathcal{J}_R(p,q,\mu,\nu)\,.
\end{equation}
\end{theorem}

\begin{proof}
We first treat the case of the soft endpoint constraint.
\par
\textbf{First step:}
Since the functional $\mathcal{J}_R$ is continuous for the weak topology (Proposition \ref{th:Continuity}), it is sufficient to prove that the space $$\mathcal{C}_0^{0} \eqdef \left\{ (p,q,p^2, \dot{p}^2) \, : \, (p,q) \in \mathcal{C}^0 \right\} \subset \mathcal{C}_R^0$$ is dense in $\mathcal{C}_R^0$ for the weak topology. Since $\nu$ is a Radon measure, it is inner regular and therefore, the measure $\tilde{\nu}_t$ defined by $0$ on $[0,t[ \times [0,1]$ and on $[t,1] \times [0,1]$ by the restriction of $\nu$ to $[t,1]\times [0,1]$ weakly converges to $\nu$ when $t$ goes to $0$. By time integration, $\int_0^s \tilde{\nu}_t \ud s' $ weakly converges to $\mu$ since $\mu(0) = 0$. In particular, it implies that the desired property can be shown only on the space of measures $\mu,\nu$ that vanish on $[0,t]$.
Let us choose a time $t \in ]0,1[$ and a couple of measures $(\mu,\nu)$ whose support is contained in $]t,1]$.
Let $(p,q,\mu, \nu) \in \mathcal{C}_R^0$ and $\varphi(t)$ be the path of diffeomorphisms associated with $(p,q)$ given by Proposition \ref{Eq:PropositionChangeOfVariable}.
\par
\textbf{Second step:} We first show that we can transform the problem to a similar condition to $p(0),q(0)\in L^{\infty}([0,1])$. 
\par

Since the path $\varphi(t)$ lies in  $H^2([0,1] , \on{Diff}^2_0([0,1]))$, for any $\varepsilon > 0$, there exists a smooth kernel in space denoted $\rho_\varepsilon$, such that the curve $\varphi_\varepsilon \eqdef \rho_\varepsilon \star \varphi$ satisfies  
$$\| \varphi_\varepsilon - \varphi \|_{H^2([0,1] , \on{Diff}^2_0([0,1]))} \leq \varepsilon\,$$
and $\varphi_\varepsilon,\dot{\varphi}_\varepsilon \in C^{\infty}([0,1])$.
The corresponding estimates on $T^*Q$ are, for a positive constant $c_1$
\begin{align*}
&\| p - p_\varepsilon \|_{H^1([0,1] , L^2([0,1]))} \leq c_1\, \varepsilon \,,\\
&\| q - q_\varepsilon \|_{H^2([0,1] , L^2([0,1]))} \leq c_1\, \varepsilon\,.
\end{align*}
Now, since $H^1([0,1],H_0^2([0,1]))$ is embedded in $C^0([0,1],H_0^2([0,1]))$, the evaluation at time $t$ is continuous and therefore, for any $t \in [0,1]$, we have, for a positive constant $c_2$
\begin{align*}
&\| p(t) - p_\varepsilon(t) \|_{L^2([0,1])} \leq c_2\, \varepsilon \,,\\
&\| q(t) - q_\varepsilon(t) \|_{L^2([0,1])} \leq c_2\, \varepsilon\,.
\end{align*}
We now use Lemma \ref{Th:LocalCubics} to define a curve close to $\varphi$ in the strong topology in $H^2([0,1],H_0^2([0,1]))$ and which has the same initial conditions $(q(0),\dot{q}(0))$:
\par
For any $\delta,t >0$ sufficiently small, there exists a choice of $\varepsilon$ such that, for the given initial condition $(q_\varepsilon(t),\dot{q}_\varepsilon(t))$, Lemma \ref{Th:LocalCubics} gives existence and uniqueness of $u_1,u_2 \in T_{q(0)}Q$ satisfying $\| u_1 \|\leq \delta$ and $\| u_2 \|\leq \delta$ and $R_q(t,u_1,u_2) = (q_\varepsilon(t),\dot{q}_\varepsilon(t) = \eta(q_\varepsilon)p_\varepsilon(t))$. Then, we define the path $\tilde{q}_\varepsilon(s)$ by gluing $R_q(s,u_1,u_2)$ for $s\in[0,t]$ and $q_\varepsilon(s)$ on $s \in [t,1]$. Since the curves and their first derivatives coincide at time $t$, the curve $\tilde{q}_\varepsilon(s)$ lies in $H^2([0,1],H_0^2([0,1]))$. In addition, this path is $O(\varepsilon,\delta)$ close to the initial curve $q(t)$ since by construction it is the case on $[t,1]$ and on $[0,t]$ by the remark \ref{Rem:Norm}. In conclusion, the constructed curve $\tilde{q}_\varepsilon(s)$ satisfies the same initial conditions than $q(t)$ at time $0$ and $\tilde{p}_\varepsilon(t),\tilde{q}_\varepsilon(t)\in C^{\infty}([0,1])$, which gives in particular the desired result $\tilde{p}_\varepsilon(t),\tilde{q}_\varepsilon(t) \in L^{\infty}([0,1])$.
\par
\textbf{Third step:} 
Since $D$ is compact, every bounded balls of  $\mathcal{C} \times \mathcal{M}(D)^2$ is metrizable. We will use such a distance denoted by $\on{dist}$ on a ball that will be defined afterwards.
Let us denote by $\rho_\varepsilon$ a smoothing kernel in space, we have, for every $f \in \mathcal{D}$
\begin{equation*}
\on{FR}_f(\rho_\varepsilon \star (\mu - p^2),\partial_t (\rho_\varepsilon \star (\mu - p^2))) \leq \langle \rho_\varepsilon \star (\nu - \dot{p}^2) ,f\rangle\,.
\end{equation*} 
Define 
\begin{align*}
\mu_\varepsilon &\eqdef \rho_\varepsilon \star (\mu - p^2) + p_\varepsilon^2\,,\\
\nu_\varepsilon &\eqdef \rho_\varepsilon \star (\nu - \dot{p}^2) + \dot{p}_\varepsilon^2\,.
\end{align*}
By choosing adequatly the smoothing kernel, we shall impose $\on{dist}(\rho_\varepsilon \star (\mu - p^2), \mu - p^2) \leq \varepsilon$ and similarly, $\on{dist}(\rho_\varepsilon \star (\nu - \dot{p}^2), \nu - \dot{p}^2) \leq \varepsilon$. Moreover, by the control we have on $p_\varepsilon,q_\varepsilon$, there exists a positive constant $c'$ such that
\begin{align*}
| \langle p_\varepsilon^2 - p^2,f\rangle | &\leq c' \varepsilon |f|_\infty\,,\\
| \langle \dot{p}_\varepsilon^2 - \dot{p}^2,f\rangle | &\leq c' \varepsilon |f|_\infty\,.
\end{align*}
Therefore, for an other positive constant $c''$, it holds
$$\on{dist}\left( (p,q,\mu, \nu), (p_\varepsilon,q_\varepsilon,\mu_\varepsilon , \nu_\varepsilon  )\right)\leq c'' \varepsilon\,.$$

Now, Lemma \eqref{th:FirstStepDensity} hereafter proves that there exists an element $ (\tilde{p},\tilde{q},\tilde{p}^2, \tilde{q}^2)$ which is $\varepsilon$ close to $ (p_\varepsilon,q_\varepsilon,\mu_\varepsilon , \nu_\varepsilon  )$ in $\mathcal{C}_R^{0}$. By the triangle inequality, we have
$$\on{dist}\left( (p,q,\mu, \nu),(\tilde{p},\tilde{q},\tilde{p}^2, \tilde{q}^2)\right)\leq (1+c'') \varepsilon\,,$$
which ends proving the result. Note that all the terms are bounded in the weak-* topology so that one can fix a ball on which the previous computations can be made and on which the distance $\on{dist}$ can be defined.
\end{proof}

\begin{remark}\label{Re:OpenQuestion2}
The relaxation of the functional given in Definition \ref{Def:AccelerationFunctional} is more involved due to the boundary constraint and is out of the scope of the paper. In fact, the strategy developed in Lemma \ref{th:FirstStepDensity} for finding the relaxation domain does not apply when adding the boundary constraint at time $1$. 
\end{remark}

\begin{lemma}\label{th:L1Approximation}
Let $\nu,\mu \in L^1(D)$ and $\mu \in W^{1,1}(D)$ that satisfy the constraint  \eqref{eq:InequalityCondition} and such that $\mu(t=0,x) = 0$ a.e in $x$. Then there exists $p_n \in H^1([0,1],L^2(D))$ such that $p_n\rightharpoonup 0$, $p_n^2 \rightharpoonup \mu$ and $\dot{p}_n^2 \rightharpoonup \nu$ and $p_n(0) = 0$.
\par
If, in addition, $\mu \in L^\infty(D)$, then we also have that $p_n \in L^\infty(D)$ and is uniformly bounded (w.r.t. $n$) in $L^\infty(D)$. 
\end{lemma}

%
\begin{lemma}\label{th:FirstStepDensity}
Let $p(0),q(0) \in L^\infty([0,1])$ and $(p(t),q(t))$ be a solution to $\dot{q} = \eta(q)\pi_q^*(p) = \eta(q)p$ such that $(p,q) \in H^1([0,1],L^\infty([0,1])) \times H^2([0,1],L^\infty([0,1]))$. For $(p,q,\mu,\nu) \in \mathcal{C}_R \cap (\mathcal{C}^0 \times L^\infty(D) \times L^1(D))$, there exists $(p_n,q_n) \in \mathcal{C}^{0}$ such that $(p_n,q_n,p_n^2,\dot{p}_n^2) \rightharpoonup (p,q,\mu,\nu)$. 
\end{lemma}
The proof of the three previous lemmas are given in Appendix \ref{Ap:Proofs}.
\ifcomments
\begin{remark}[Open question]
We do not know if the previous result holds when adding the boundary constraint at time $1$, namely, $(p_n(1),q_n(1)) = (p(1),q(1))$ and $\mu(1)=p(1)^2$. This would answer to the open question in remark \ref{Re:OpenQuestion1} and \ref{Re:OpenQuestion1}.
\end{remark}
\fi

\section{Variational study of the relaxed acceleration}\label{Sec:MainTheorem}
Since we aim at minimizing the relaxed functional $\mathcal{J}_R$, we give a first simple reduction by minimizing over $\nu$. The constraint $\on{FR}(\mu - p^2,\partial_t(\mu - p^2)) \preccurlyeq \nu - \dot{p}^2$ gives that the first term $\int_D\! \eta \ud \nu$ is lower bounded by $ \on{FR}_\eta(\mu - p^2,\partial_t(\mu - p^2)) + \eta \dot{p}^2$ and the formula \eqref{eq:FirstVariationalReduction} follows. 
Minimizing $\mathcal{J}_R$ over $\nu$ when $(p,q,\mu)$ are fixed, gives
\begin{multline}\label{eq:FirstVariationalReduction}
\min_{\nu} \mathcal{J}_R = \on{FR}_\eta(\mu - p^2,\partial_t(\mu - p^2))  \\+ \!\int_{D} \! \eta(q)  \Big(\dot{p}  + \mathcal{U}(\mu,q) -  \pi_1(p,q) - \pi_2(\mu,q)\Big)^2 \! \ud x \!\ud t + P(p(1),q(1))\,.
\end{multline}
This way of writing the functional was used for proving the lower semicontinuity and since it is now proven, we can use a slightly simpler and more geometric formulation.
Using the defect measure $\Delta \eqdef \mu - p^2$ and recognizing the acceleration of the original curve $(p,q)$, the functional can be written as $\mathcal{F}(\Delta, p, q) \eqdef \min_{\nu} \mathcal{J}_R(\mu, \nu,p,q)$ with
\begin{multline}\label{eq:ReducedFunctional}
\mathcal{F}(\Delta, p, q) =  \on{FR}_\eta(\Delta,\partial_t \Delta)  +  \int_0^1 \left\| \left(\frac{D}{Dt}\dot{q}\right)^{\flat} + \pi_q^*(\mathcal{U}(\Delta,q)) \right\|^2_q\,\ud t + P(p(1),q(1))\,.
\end{multline}
Note that the boundary constraint on $\mu$ transfers to $\Delta$ as $\Delta_{t = 0} = 0$ (the disintegration of $\Delta$ at time $t=0$), which will be written $\Delta \in \mathcal{M}_0(D)$. Therefore, we are interested in the minimization of the functional $\mathcal{F}$ with respect to $\Delta$ taking values in $\mathcal{M}_0(D)$ space of measures 

The previous functional can be understood as follows: The original acceleration functional can be possibly made lower using the term $\mu$ which is paid at the price $\on{FR}_\eta(\mu - p^2,\partial_t(\mu -p^2))$. In particular, when the path $(p(t),q(t)) \in T^*Q$ is fixed, one can look at the minimization over $\Delta$ of the previous functional. With respect to $\Delta$ the previous functional is convex and thus any critical point w.r.t. $\Delta$ is a global minimizer. 
A priori, the functional $\mathcal{F}$ is not strictly convex which is due to the fact the Fisher-Rao functional is one homogeneous. Moreover the projection operator $\pi_q^*$ has a non trivial kernel and therefore the strict convexity of the norm squared does not help. However, we do not need strict convexity to conclude to the following interesting property. The result will be given by the following simple lemma whose proof is postponed in Appendix \ref{Ap:Proofs}.

\begin{lemma}\label{Th:LemmaSimpleConvexity}
Let $X$ be a Banach space, $f:X\mapsto \R_+$ be a convex function and $x_0 \in X$ such that $f^{-1}(\{f(x_0)\}) = \{x_0\}$. Let $g:X \mapsto \R$ be a strictly convex function and $A$ be a linear operator on $X$, then if $x_0$ is a minimum for the function $f+g \circ A$, it is the unique minimum.
\end{lemma}

\begin{proposition}\label{Th:WeakConvexity}
Let $(p,q) \in \mathcal{C}^0$ be a path in $T^*Q$. The minimization of $\mathcal{F}$ over $\Delta$ on $\mathcal{M}_0(D)$ is a convex minimization problem and if the functional $\mathcal{F}$ has a minimum at $\Delta=0$ then it is the unique minimum. 
\end{proposition}

\begin{proof}
This is the direct application of the previous lemma if we show that the equality $\on{FR}_\eta(\Delta,\partial_t \Delta)=0$ implies that $\Delta \equiv 0$ since $\Delta(0) =0$. Suppose on the contrary that $\Delta \neq 0$, then, for a regularizing kernel $\rho$ we have that $\rho \star \Delta \neq 0$ and by convexity, 
\begin{equation*}
\on{FR}_\eta(\rho \star \Delta,\partial_t(\rho \star \Delta)) \leq \on{FR}_{\rho \star \eta}(\Delta,\partial_t\Delta)\,,
\end{equation*}
moreover, by continuity of $\eta$, there exists a constant $C>0$ such that $\rho \star \eta \leq C \eta$ and therefore $\on{FR}_{\rho \star \eta}(\Delta,\partial_t\Delta) \leq C \on{FR}_{\eta}(\Delta,\partial_t\Delta)=0$.
Thus, $\on{FR}_\eta(\rho \star \Delta,\partial_t(\rho \star \Delta))=0$ which directly implies that $\rho \star \Delta=0$ which is a contradiction.
\end{proof}

In the next theorem, we make explicit the first-order optimality condition.

\begin{theorem}\label{Th:FirstOrderCondition}
Let $(p,q) \in \mathcal{C}^0$ be a path in $T^*Q$ and consider the minimization of $\mathcal{F}$ over $\Delta$  on the relaxed set $\mathcal{C}_R^{0}$. A necessary and sufficient condition for optimality at $\Delta$ is that 
both following conditions are satisfied:
\\
For any $\mu \in \mathcal{M}_0(D)$ such that $\partial_t \mu \in \mathcal{M}(D)$,
\begin{equation}\label{Eq:FirstInequalityCharacterization}
\on{FR}_{\eta}(\mu,\partial_t \mu) \geq \langle \mu,  -w \rangle\,,
\end{equation}
 \\
and the equality
 \begin{equation*}
\on{FR}_{\eta}(\Delta,\partial_t \Delta) +  \langle \Delta,  w \rangle = 0\,,
 \end{equation*}
 where \begin{equation*}
w =  \eta \int_0^x \left(\frac{D}{Dt}\dot{q} + \pi^*(\mathcal{U}(\Delta,q)) \right) \ud y\,.
\end{equation*}

\end{theorem}

The proof of this theorem relies on a well-known lemma  for one-homogeneous functionals, whose proof is given in Appendix \ref{Ap:Proofs}.
\begin{lemma}\label{Th:GeneralLemma}
Let $X$ be a Banach space, $X^*$ be its topological dual and $f: X \mapsto \R \cup \{ +\infty \}$ be a convex function which is positively one-homogeneous. That is $f(\lambda x) = \lambda f(x)$ for every $\lambda \geq 0$. Then, 
$v \in \partial f(x)$ if and only if the following two conditions hold
\\
(1) $f(y) \geq \langle y , v \rangle$ for every $y \in X$,
\\ (2) $f(x) = \langle x , v \rangle$.
\end{lemma}

\begin{proof}[Proof of Theorem \eqref{Th:FirstOrderCondition}]
The first-order optimality condition is $0 \in \partial \mathcal{F}(\Delta)$.
Note that the functional $\mathcal{F}$ can be written as $\mathcal{F}(\Delta,p,q) = \mathcal{F}_1(\Delta) + \mathcal{G}(\Delta)$ (omitting the dependency in $(p,q)$)
where 
\begin{equation*}
\mathcal{G}(\Delta) \eqdef  \int_0^1 \left\| \left(\frac{D}{Dt}\dot{q}\right)^{\flat} + \pi^*(\mathcal{U}(\Delta,q)) \right\|^2_q\,\ud t + P(p(1),q(1))
\end{equation*}
is a smooth function on the space of measures $\mathcal{M}$ and $\mathcal{F}_1$ is a one-homogeneous function on $\mathcal{M}_0(D)$. Namely we have
\begin{equation*}
\mathcal{F}_1(\Delta) = 
\begin{cases}
\on{FR}(\Delta,\partial_t \Delta) \text{ if } \partial_t \Delta \in \mathcal{M} \\
+\infty \text{ otherwise.} 
\end{cases}
\end{equation*}
Therefore, the first-order condition for optimality reads $0 \in \partial \mathcal{F}_1(\Delta)  + D\mathcal{G}(\Delta)$, by smoothness of $\mathcal{G}$, see \cite[Proposition 5.6]{EkelandTemam}. Then, Lemma \ref{Th:GeneralLemma} gives the result using 
\begin{equation*}
D\mathcal{G}(\Delta) =  \eta \int_0^x \left(\frac{D}{Dt}\dot{q} + \pi^*(\mathcal{U}(\Delta,q)) \right) \ud y\,,
\end{equation*}
and the first-order condition $0 \in \partial \mathcal{F}$ then reads
\begin{equation}\label{Eq:FirstOrderConditionSimple}
- D\mathcal{G}(\Delta) \in \partial \mathcal{F}_1(\Delta)\,.
\end{equation}
\end{proof}

This characterization will be used to show the existence of standard paths $(p,q)\in \mathcal{C}$ such that the minimization of $\mathcal{F}$ with respect to $\Delta$ is not $\Delta = 0$. However, this characterization is still relatively general and it only uses the one-homogeneity of the Fisher-Rao functional.
In the next proposition, we give a more explicit but only sufficient condition for optimality. 
Recall that we denoted by $C_{1,0}(D)$ the space of continuous real functions which are $C^1$ with respect to the first (time) variable.

\begin{proposition}\label{Th:FirstOrderConditionSufficient}
Let $(p,q) \in \mathcal{C}$ be a path in $T^*Q$. A sufficient condition for optimality at $\Delta$ for the minimization of $\mathcal{F}$ over $\Delta \in \mathcal{M}_0(D)$ is that 
\par (1) there exists a function $g \in C_{1,0}(D)$ such that $g(1,x) = 0$ for all $x \in [0,1]$,
\begin{equation}\label{eq:RiccatiInequality}
\partial_tg + \frac{1}{\eta} g^2 \leq \eta \int_0^x \left(\frac{D}{Dt}\dot{q} + \pi^*(\mathcal{U}(\Delta,q)) \right) \ud y\,,
\end{equation}
\par (2) the following equality is satisfied
\begin{equation}
\langle - D\mathcal{G}(\Delta),\Delta \rangle = \on{FR}_\eta(\Delta,\partial_t \Delta)\,.
\end{equation}
\end{proposition}

\begin{proof}
Instead of considering the functional defined on the space of measures $\mathcal{M}$, we consider its extension to $\mathcal{M}^2(D)$ defined as follows (which we still denote by $\mathcal{F}$)
\begin{equation*}
\mathcal{F}(\Delta,\delta) = \on{FR}_\eta(\Delta,\delta) + \mathcal{G}_1(\Delta,\delta) + \iota_{V}(\Delta,\delta)\,,
\end{equation*}
where $\mathcal{G}_1(\Delta,\delta)\eqdef \mathcal{G}(\Delta)\,$ and 
\begin{equation}\label{Eq:DefofV}V\eqdef \{ (\Delta,\delta) \in \mathcal{M}^2 \, : \,  \langle \partial_t f , \Delta \rangle + \langle f,\delta  \rangle= 0 \text{ } \forall f \in C_{1,0}(D) \text{ and } f(1,x) = 0 \, \forall x \in [0,1] \,\}\,,
\end{equation}
The definition of $V$ corresponds to the weak definition of the space of measures $(\Delta, \delta)$ such that $\Delta(0) = 0$ and $\partial_t \Delta = \delta$.
Then, the first-order condition reads 
\begin{equation}\label{Eq:Condition1}
0 \in \partial \left(\on{FR}_\eta(\Delta,\delta) + \mathcal{G}_1(\Delta,\delta) + \iota_{V}(\Delta,\delta)\right)\,.
\end{equation}
Since $\mathcal{G}_1$ is a smooth function, $\partial \mathcal{G}_1(\Delta,\delta) = (D\mathcal{G}(\Delta),0)$ and  condition \eqref{Eq:Condition1} becomes 
\begin{equation}\label{Eq:Condition2}
(-D\mathcal{G}(\Delta),0) \in \partial \left(\on{FR}_\eta(\Delta,\delta) + \iota_{V}(\Delta,\delta)\right)\,.
\end{equation}
However, standard conditions in \cite{EkelandTemam} or finer conditions in \cite[Theorem 7.15]{RaduBot} cannot be applied to guarantee that the subdifferential of the sum
is the sum of the subdifferentials for the sum $\on{FR}_\eta(\Delta,\delta) + \iota_{V}(\Delta,\delta)$. We only have
\begin{equation*}
\partial \on{FR}_\eta(\Delta,\delta) + \partial \iota_{V}(\Delta,\delta) \subset \partial \left(\on{FR}_\eta(\Delta,\delta) +  \iota_{V}(\Delta,\delta) \right)\,.
\end{equation*}
Using the definition of $V$ in equation \eqref{Eq:DefofV}, it is clear that $$\left\{ (\partial_tf,f) \,: \, f \in C_{1,0}(D) \text{ and } f(1,x) = 0 \, \,\forall x \in [0,1] \,\right\} \subset \partial \iota_V(\Delta,\delta)\,.$$
Using Proposition \ref{Th:SubDifOfFR}, we have that 
\begin{equation*}
\left\{ (u,v) \in C^0(D) \, : \, H_\eta(u,v) = 0 \text{ and } \langle u, \Delta \rangle +  \langle v, \delta \rangle = \on{FR}_\eta(\Delta,\delta) \right\} \subset \partial \on{FR}_\eta(\Delta,\delta)\,.
\end{equation*}

%
%
%
%
Therefore, a sufficient condition for optimality is the existence of $(u,v) \in C(D)^2$ and $f \in C_{1,0}(D)$ such that $f(1,x)=0$ for all $x \in [0,1]$
\begin{equation}\label{Eq:FirstCondition1}
(-D\mathcal{G}(\Delta),0) = (u,v) +(\partial_tf,f)\,,
\end{equation}
satisfying also 
\begin{equation*}
\langle \partial_t f - D\mathcal{G}(\Delta),\Delta \rangle + \langle f,\partial_t \Delta \rangle = \on{FR}_\eta(\Delta,\partial_t \Delta)
\end{equation*}
which is condition (2) when simplifying the formula by integration by part on $f$ and $\Delta$.
Equation \eqref{Eq:FirstCondition1} can be rewritten $H_\eta(-\partial_t f-D\mathcal{G}(\Delta),- f)=0$ or equivalently, with $g= -f$,
\begin{equation*}
 \exists \, g \in C_{1,0}(D) \text{ s.t. } \partial_tg + \frac{1}{\eta} g^2 \leq \eta \int_0^x \frac{D}{Dt}\dot{q} + \pi^*(\mathcal{U}(\Delta,q))  \ud y\,,
\end{equation*}
which is condition (1).
\end{proof}

\begin{remark}
Note that equality \eqref{Eq:FirstCondition1} implies that $-D\mathcal{G}(\Delta)$ is a continuous function on $D$. Since we have that 
\begin{equation*}
D\mathcal{G}(\Delta) =  \eta \int_0^x \left(\frac{D}{Dt}\dot{q} + \pi^*(\mathcal{U}(\Delta,q)) \right) \ud y\,,
\end{equation*}
$ \int_0^x \frac{D}{Dt}\dot{q} \ud y$ is only $L^2([0,1],H^1)$ and therefore it is not continuous on $D$.
However, the second term $\int_0^x \pi^*(\mathcal{U}(\Delta,q)) \ud y$ is more regular. Since $\Delta$ is a minimizer, Proposition \ref{Th:TimeRegularity} applies and therefore $\int_0^x \pi^*(\mathcal{U}(\Delta,q)) \ud y$ is Hölder continuous in time with values in $C^1$ w.r.t. $x$ so that it is actually jointly continuous on $D$. 
\par
Now, assuming that $(p,q,\Delta)$ is a minimizer of the relaxed acceleration functional, it would be probably possible to prove time regularity as is usual for $(p,q)$ using the Euler-Lagrange equation and therefore $D\mathcal{G}(\Delta)$ would be continuous.
\end{remark}

Since we will be particularly interested in the case where $\Delta=0$, we emphasize it hereafter.
\begin{corollary}
Let $(p,q) \in \mathcal{C}$ be a path in $T^*Q$ such that $ \eta \int_0^x \frac{D}{Dt}\dot{q} \ud y \in C^0(D)$. The minimization of $\mathcal{F}$ over $\Delta$ for fixed $(p,q) \in \mathcal{C}$ is attained at $\Delta= 0$ if 
 there exists a function $g \in C_{1,0}(D)$ such that 
\begin{equation}\label{eq:RiccatiInequalitySimple}
\partial_tg + \frac{1}{\eta} g^2 \leq \eta \int_0^x \frac{D}{Dt}\dot{q} \ud y\,.
\end{equation}
\end{corollary}

We are thus interested in global solutions of the Riccati inequality
\begin{equation*}
\partial_tg + \frac{1}{\eta} g^2 \leq \eta \int_0^x \frac{D}{Dt}\dot{q} \ud y
\end{equation*}
on the time interval $[0,1]$. Riccati equations have been intensively studied and we summarize some useful results hereafter.

\begin{lemma}[Sturm comparison lemma]
Let $m,M$ be two functions in $L^1(D)$ such that $m\leq M$ and $\eta_1,\eta_2 \in C^0(D)$ be positive functions such that $\eta_1\leq \eta_2$. Consider the two Riccati equations
\begin{equation}\label{eq:FirstRiccati}
\partial_tg + \frac{1}{\eta_1} g^2 = m
\end{equation}
and 
\begin{equation}\label{eq:SecondRiccati}
\partial_tf+ \frac{1}{\eta_2} f^2 = M\,.
\end{equation}
Then, if $g(t=0,\cdot) \leq f(t=0,\cdot)$, the solution $f$ to \eqref{eq:SecondRiccati} exists at least as long as the solution $g$ to \eqref{eq:FirstRiccati} exists.
\end{lemma}

\begin{proof}
Substract Equations \eqref{eq:SecondRiccati} and \eqref{eq:FirstRiccati} to get, for $y \eqdef f-g$
\begin{equation*}
\partial_t y + \frac{1}{\eta_2}y(f+g) = M-m + \left(\frac{1}{\eta_1}- \frac{1}{\eta_2}\right) g^2  \,,
\end{equation*}
which implies that
\begin{equation*}
y(t)E(t) = y(0) + \int_0^t \left(M(s)-m(s) + \left(\frac{1}{\eta_1(s)}- \frac{1}{\eta_2(s)}\right) g(s)^2 \right)\, \ud s\,,
\end{equation*}
where $E(t) \eqdef \exp\left(\int_0^t \frac{1}{\eta}(f+g)\,\ud s\right)$.
Since $y(0) \geq 0$ and $M(s)-m(s) \geq 0$ and $\frac{1}{\eta_1(s)}- \frac{1}{\eta_2(s)} \geq 0$, we have that $f(t,\cdot) \geq g(t,\cdot)$.
Moreover, we also have that $\partial_t f \leq M$ which implies that $f$ is bounded on every time interval of existence of $g$. Therefore, $f$ exists as long as $g$ exists.
\end{proof}

\begin{corollary}
A solution to the inequality \eqref{eq:RiccatiInequality} satisfying also boundary conditions exists if and only if there exists a solution $g\in C_{1,0}(D)$ to 
\begin{equation}\label{eq:RiccatiEquality1}
\partial_tg + \frac{1}{\eta} g^2 = \eta \int_0^x \frac{D}{Dt}\dot{q} \ud y\,,
\end{equation}
and satisfying the  constraint $g(1,x) = 0$ for $x \in [0,1]$;
Equivalently, iff there exists a solution $u$ to
\begin{equation}\label{eq:RiccatiEquality2}
\partial_{tt}u+ \left(\int_0^x\dot{q} \ud y \right)\partial_tu - \left( \int_0^x \frac{D}{Dt}\dot{q} \ud y\right)u=0  \,,
\end{equation}
such that for every $x\in [0,1]$, $u(0,x)=0$ and $u(t,x)>0$ for $t \in ]0,1]$ and $\partial_tu(1,x) = 0$ for $x \in [0,1]$.
The previous equation can be rewritten as follows
\begin{equation}\label{eq:RiccatiEquality3}
-\partial_t\left( \eta \partial_tu \right) + \eta \left( \int_0^x \frac{D}{Dt}\dot{q} \ud y\right)u=0  \,.
\end{equation}
\end{corollary}

\begin{proof}
The comparison lemma implies that if the solution to the Riccati inequality exists then the solution to the Riccati equation  also exists.
The equivalence with the second-order linear ordinary differential equation is obtained by exchanging the function $g$ with $\frac{\partial_tu}{u}$.
\end{proof}
Let us stress the fact that there exist coefficients of the Riccati for which there does not exist any solution defined on the whole interval $[0,1]$ as detailed in the following remark.
\begin{remark}\label{eq:SimpleRiccati}
Let $a,b$ be two non-negative real numbers, then there exists a solution to the Riccati equation
\begin{equation*}
\dot{x}+ a^2 x^2 = -b^2\,
\end{equation*}
defined on the time interval $[0,1]$ if and only if $ab < \pi$.
By direct integration, we have, for a solution defined on the time interval $[0,1]$, assuming $b\neq0$ and $t\leq 1$
\begin{equation}\label{Eq:ExplicitSolution}
\int_{x(0)}^{x(t)} \frac{\ud y}{b^2+a^2y^2} = \frac{1}{ab}(\tan^{-1}(x(0))-\tan^{-1}(x(t)))  = \int_0^t \ud s \leq 1\,.
\end{equation} 
This gives the result.
For $b=0$, $x \equiv 0$ is a solution. In this particular case, existence of solutions implies existence of solutions satisfying the boundary condition $x(1) = 0$.
\end{remark}

\begin{corollary}
The Riccati equation \eqref{eq:RiccatiEquality1} has a solution defined on $D$ vanishing at time $t=1$ if 
\begin{equation}\label{Eq:FirstCondition}
\frac{\sup_D [\eta \int_0^x \frac{D}{Dt}\dot{q} \ud y ]_-}{\inf_D \eta} < \pi^2\,,
\end{equation} 
where $[f]_-$ denotes the negative part of $f$.

The Riccati equation \eqref{eq:RiccatiEquality1} has no solution defined on $D$ if there exists $x \in [0,1]$ such that $\eta \int_0^x \frac{D}{Dt}\dot{q} \ud y \leq -m$ for a positive constant $m$ and if
\begin{equation*}
\frac{m}{\sup_{t \in [0,1]} \eta} \geq \pi^2\,.
\end{equation*} 
\end{corollary}

\begin{proof}
This is a direct consequence of the remark \ref{eq:SimpleRiccati} and the Comparison lemma.
\end{proof}

As a direct but important application, we have:
\begin{corollary}
The set of paths $(p,q) \in \mathcal{C}$ such that $\mathcal{F}$ has an optimum w.r.t. $\Delta$ at $\Delta =0$ contains an open neighborhood of every geodesics $q(t)$ in the $C^2([0,1],L^2([0,1]))$ topology.
\end{corollary}
\begin{proof}
For geodesics, Equation \eqref{Eq:FirstCondition} is trivially satisfied since $\frac{D}{Dt}\dot{q}=0$ and this condition is stable under perturbations in $C^0([0,1],L^2([0,1]))$. 
\end{proof}

We now want to show that there exist paths for which $\mathcal{F}$ does not achieve its minimum at $\Delta =0$. A simple situation where the acceleration term is easily computable is a reparametrized geodesic.
Thus, we are now interested in paths that are reparametrizations of a geodesic which are the simplest critical points $(p,q) \in T^*Q$ for the acceleration functional.
Recall that for a geodesic $q(t)$ and $\alpha$ a time reparametrization, the acceleration of $y \eqdef q \circ \alpha$ satisfies we have that $\frac{D}{Dt}\dot{y} = \ddot{\alpha} \dot{q}(\alpha)$.

\begin{proposition}
There exist $\alpha$ a time reparametrization and a geodesic denoted in Hamiltonian coordinates by $(p_0,q_0)$ such that $(p(t),q(t)) \eqdef (\dot{\alpha}p_0(\alpha),q_0(\alpha(t)))$ is a curve in $\mathcal{C}$ for which $\Delta=0$ is not a minimum of the functional $\mathcal{F}(\Delta)$.
\end{proposition}

%

\begin{proof}
We first define the geodesic for which we will then construct a suitable reparametrization $\alpha$.
Consider a momentum $m_0$ which is symmetric with respect to $1/2$. For instance, $m_0 = \delta_{1/4} - \delta_{3/4}$ and we consider the geodesic generated at identity by $m_0$. From Proposition \ref{Th:ExampleOfGeodesic}, we have  that $\eta(t,1/2)$ is decreasing and that $\lim_{t \to \infty} \eta(t,1/2) =0$.
Actually, any geodesic that has a fixed point $z \in ]0,1[$ and a jacobian $\eta(t,z)$ whose limit is $0$ for $t\to +\infty$, will allow for the following construction. 
\par
Let us now prove the existence of a suitable reparametrization. Using Theorem \ref{Th:FirstOrderCondition}, it is sufficient to prove that there exists $\mu \in \mathcal{M}(D)$ such that the following inequality is satisfied:
\begin{equation}\label{Eq:Contradiction}
\on{FR}_{\eta}(\mu,\partial_t \mu) < -\left\langle \mu, \eta \int_0^x \frac{D}{Dt}\dot{q}  \ud y  \right\rangle\,,
\end{equation}
which corresponds to the inequality \eqref{Eq:FirstInequalityCharacterization} where $\Delta = 0$.
We consider $\mu(t) = f(t)^2\delta_{1/2}$ where $\delta_{1/2}$ is the Dirac distribution at point $1/2$ and the previous inequality can be written as
\begin{equation}\label{Eq:CounterExample}
\int_0^1 \dot{f}^2(s) \eta(s,1/2) \ud s < -\int_0^1 f^2(s)\eta(s,1/2) \int_0^{1/2} \frac{D}{Dt}\dot{q}(s,y)  \ud y \ud s\,.
\end{equation}
We now observe that we have the following formula $$\eta(s,1/2) \int_0^{1/2} \frac{D}{Dt}\dot{q}(s,y)  \ud y=\eta(s,1/2) \int_0^{1/2}\ddot{\alpha}(s) \dot{q}_0(\alpha(s),y)  \ud y = \frac{\ddot{\alpha}}{\dot{\alpha}} \frac{\ud }{\ud t} [\eta_0 \circ \alpha]\,,$$
where the second equality comes from
$\frac{\ud}{\ud t} [\eta(q)(t,x)] = \eta(t,x) \int_0^x \dot{q}(s,y) \ud y $.
Inequality \eqref{Eq:CounterExample} can be reformulated as 
\begin{equation}\label{Eq:CounterExample2}
\int_0^1 \dot{f}^2(s) \eta(s) \ud s < -\int_0^1 f^2(s)\frac{\ddot{\alpha}}{\dot{\alpha}} \dot{\eta}(s,1/2) \ud s\,,
\end{equation}
where we have used $\eta(s,1/2) = \eta_0(\alpha(s),1/2)$. Since $\eta(s,1/2) \leq 1$, we have
\begin{equation*}
\int_0^1 \dot{f}^2(s) \eta(s) \ud s \leq \int_0^1 \dot{f}^2(s)\ud s\,.
\end{equation*}
Let us consider for instance $f(t) = \sin(\pi t)$, one has
$
\int_0^1 \dot{f}^2(s) \ud s = \frac{\pi^2}{2}\,.
$
On the interval $[\frac{1}{4},\frac{3}{4}]$, one has $\sin^2(\pi t) \geq \frac{1}{2}$ and therefore, since $\dot{\eta}(s,1/2)\leq 0$
\begin{equation*}
-\int_0^1 f^2(s)\frac{\ddot{\alpha}}{\dot{\alpha}} \dot{\eta}(s,1/2) \ud s \geq - \frac12 \int_{1/4}^{3/4} \frac{\ddot{\alpha}}{\dot{\alpha}} \dot{\eta}(s,1/2) \ud s\,
\end{equation*}
under the additional assumption $\frac{\ddot{\alpha}}{\dot{\alpha}} \geq 0$ which will be satisfied in our subsequent choice. By \eqref{Eq:CounterExample2}, it is sufficient to prove that there exists a reparametrization $\alpha$ such that 
\begin{equation}
\int_0^1 \dot{f}^2 \ud t < -\frac 12 \int_{1/4}^{3/4} \frac{\ddot{\alpha}}{\dot{\alpha}} \dot{\eta}(s,1/2) \ud s\,.
\end{equation}
Let $A>0$ and $\alpha(t) = 1/4 + e^{At}-e^{A/4}$ for $t>1/4$, then
$\frac{\ddot{\alpha}}{\dot{\alpha}} =A>0$ for $t>1/4$ which implies
\begin{equation*}
-\int_0^1 f^2(s)\frac{\ddot{\alpha}}{\dot{\alpha}} \dot{\eta}(s,1/2) \ud s \geq \frac12 A (\eta_0 \circ \alpha(1/4) - \eta_0 \circ \alpha(3/4)) \ud s\,.
\end{equation*}
To conclude, we note that $\eta_0 \circ \alpha(3/4) - \eta_0 \circ \alpha(1/4) \geq \eta_0(3/4) - \eta_0(1/4)$ as long as $A$ is sufficiently big enough so that $\alpha(3/4)\geq 3/4$. Therefore, there exists $A$ big enough such that 
\begin{equation*}
\frac{\pi^2}{2} = \int_0^1 \dot{f}^2(s)\ud s < \frac12 A (\eta_0(3/4) - \eta_0(1/4)) \,,
\end{equation*}
which implies inequality \eqref{Eq:CounterExample}.
\end{proof}

In the previous proof, we did not look for the simplest parametrization for which the inequality \eqref{Eq:CounterExample} holds. However, since cubic reparametrizations of geodesics are critical points for the initial functional of the acceleration, it is natural to ask if there exists a cubic reparametrization of the previous geodesic for which the same estimations hold. 
However, the analytical computations are rather difficult since there is no closed form solutions of the geodesic equation and we therefore chose to show the result by numerical simulations.

\section{A numerical experiment}\label{Sec:Numerics}

In this section, our goal consists in showing numerically that there exists cubic reparametrizations of geodesics that are not minimizers with respect to defect measure. It thus proves that being a critical point of the initial functional does not help for being a minimizer w.r.t. the defect measure. Although this result is not surprising because for non convex variational problems, critical points are not necessarily minimizers, this is yet another hint that tends to show that there exist minimizing solutions to the relaxed acceleration functional for which the defect measure is not zero.

We consider the following geodesic defined by its initial position $\varphi_0 = \on{Id}$, and its initial velocity corresponding to the momentum $m_0 \eqdef 15\, (\delta_{1/4} - \delta_{3/4})$. By the result in Proposition \ref{Th:ExampleOfGeodesic}, the geodesic has a jacobian at $x=1/2$ which tends to $0$, as shown in Fig. \ref{Fig:JacobianGeodesic}. We consider the cubic reparametrization $r(t) \eqdef 2t^3$ plotted in Fig \ref{Fig:Reparametrization}. Then, the right-hand side of the Riccati equation \eqref{eq:RiccatiEquality1} is shown in Fig \ref{Fig:RiccatiRHS}.

\begin{figure}[!htb]
\minipage{0.50\textwidth}
  \includegraphics[width=\linewidth]{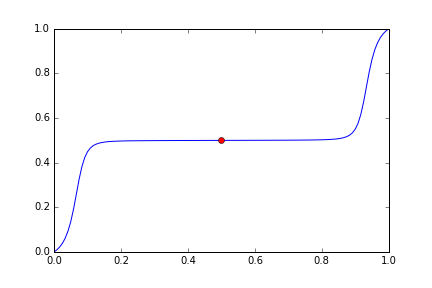}
  \caption{The diffeomorphism at time $t=16$. The red dot represents the point $1/2$ which is fixed by the flow.}\label{Fig:Diffeo}
\endminipage\hfill
\minipage{0.50\textwidth}
  \includegraphics[width=\linewidth]{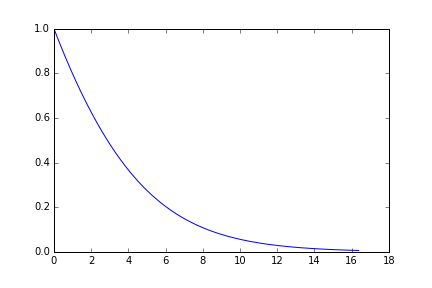}
  \caption{The Jacobian at $x=1/2$ for the initial geodesic.}\label{Fig:JacobianGeodesic}
\endminipage\vfill
\minipage{0.50\textwidth}%
  \includegraphics[width=\linewidth]{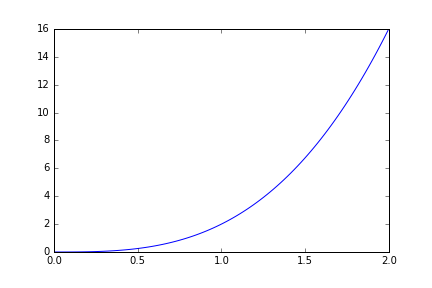}
  \caption{The cubic reparametrization defined by $r(t) = 2t^3$.}\label{Fig:Reparametrization}
\endminipage
\minipage{0.50\textwidth}%
  \includegraphics[width=\linewidth]{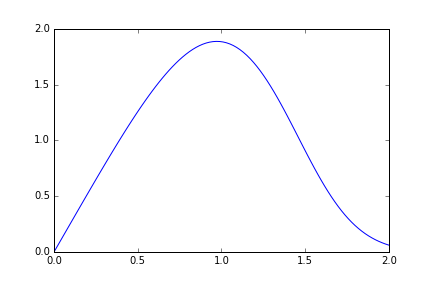}
  \caption{The right-hand side of the Riccati equation \eqref{eq:RiccatiEquality1}.}\label{Fig:RiccatiRHS}
\endminipage
\end{figure}

We then consider $\Delta(t,x) = \sin(2\pi t)^2 \delta_x$ and we numerically compute the Fisher-Rao functional $\on{FR}_{\eta}(\Delta,\partial_t \Delta) = 1.387$ and the quantity $-\int_0^1  \sin(2\pi t)^2\dot{\eta}(t,1/2) \frac{\ddot{\alpha}}{\dot{\alpha}}  \ud t = 1.483$, which satisfy the inequality \eqref{Eq:Contradiction}:
\begin{equation*}
\on{FR}_{\eta}(\Delta,\partial_t \Delta) < -\left\langle \Delta, \eta \int_0^x \frac{D}{Dt}\dot{q}  \ud y  \right\rangle\,.
\end{equation*}
It thus shows numerically that this cubic reparametrized geodesic is not a minimizer with respect to the defect measure $\Delta$ by application of Theorem \ref{Th:FirstOrderCondition}.

\section{Conclusion and perspectives}
In this paper we have studied the relaxation of the acceleration in the case of a right-invariant Sobolev metric on the group of diffeomorphisms in one dimension. We have shown, that for the relaxed endpoint constraint, the relaxation of the functional involves the Fisher-Rao functional. Being a convex and one-homogeneous functional on measures, we have derived several optimality conditions by means of convex analysis: optimality is linked to the existence of a solution to a Riccati equation that involves the acceleration of the curve.
\par
Several interesting questions remain open:
\\ 
(1) We have proved the existence of an example where the minimizer of the relaxed acceleration functional has a non zero defect measure. In other words, we ask whether all the minimizers of the relaxed acceleration functional satisfy $\Delta = 0$. \\
(2) Another important point is to add a boundary hard constraint at time $1$ on the couple $(p,q)$. It is probable that it will involve a different form of the Fisher-Rao metric. \\
(3) Our proof relied at different steps on the one dimensional case. For instance, the formulation of the geodesic equations was a key point for obtaining the properties on the derivative of the metric. This step can probably be extended to higher dimensional case, i.e. on $\on{Diff}^s([0,1]^d)$ for $s>d/2+1$. However, the relaxation of the functional will certainly differ from the Fisher-Rao functional. \\
(4) In a different direction, we could also have written the Euler-Lagrange equations for the relaxed functional on the space $\mathcal{C}_R$. An interesting question would be to use this Euler-Lagrange equation to derive regularity properties of the solutions from the endpoints such as in the geodesic case as in Corollary \ref{Th:Regularity}. \\
(5) We also have left behind a technical question about the characterization of the subdifferential of the Fisher-Rao functional similarly to what has been done for the $TV$ norm on functions. This would lead to a finer characterization than that presented in Theorem \ref{Th:FirstOrderCondition}.
\appendix

\section{Computation of the geodesic equations}\label{Ap:ComputationOfGeodesicEquations}
We give hereafter the proof of theorem \ref{Th:GeodesicEquations} which is based on rather straightforward computations.
\begin{proof}
%
 
As is well known, the geodesic equation on $Q$ is the projection of the geodesic equation on $L^2$ for the same Riemannian metric, whose Levi-Civita connection is denoted by $\nabla$.
 We denote by  $\nabla^0$ the induced Levi-Civita connection on $Q$. It is given by
\begin{equation*}
\nabla^0_{X}Y = \nabla_{X}Y - \mathbf{II}(X,Y)\,,
\end{equation*}
with $X,Y$ two vector fields tangent to $Q$ and arbitrarily extended to $L^2$.
In the previous formula, the second fundamental form $\mathbf{II}$ is a tensor defined by
$$\mathbf{II}(X,Y) = (\nabla_X Y )^\perp\,,$$
where $(\nabla_X Y )^\perp$ is the orthogonal projection of $\nabla_X Y$ on $(TQ)^\perp$.
Therefore, the geodesic equation on $Q$ can be written as 
$$ \frac D {Dt} \dot{q} = \mathbf{II}(\dot{q},\dot{q}) = \left(\frac D {Dt} \dot{q}\right)^\perp\,.$$
The previous formulation can be rewritten as $\pi_q(\frac D {Dt} \dot{q}) = 0$ where $\pi_q$ is the orthogonal projection  corresponding to the constraints w.r.t. the scalar product \eqref{Metric}, i.e the scalar product in $L^2$ with the weight $\frac{1}{\eta}$. 
If there were no constraints on $q$, the geodesic equation would be
\begin{equation}\label{GeodesicEquationL2}
\begin{cases}
\dot{q} = \eta(q)p\\
\dot{p} = -\frac 12 \int_x^1 \eta(q)p^2 \ud y\,.
\end{cases}
\end{equation}
In order to project these equations on $T_q^*Q$, we need to compute the scalar product of $\dot{p}  + \frac 12 \int_x^1 \eta(q)p^2 \ud y$ with $1$ and $\varphi$. Note that for a given path $(q(t),p(t)) \in T^*Q$, one has $\frac {\ud}{\ud t}\int_0^1 \eta p \ud x = 0$ which implies, using formula \eqref{GeodesicEquationL2} and after few (straightforward) computations
\begin{equation*}
\langle \dot{p}, 1 \rangle_\eta = \int_0^1 \dot{p} \,\eta \ud x   =- \int_0^1 p\, \eta \left(\int_0^x \eta p \ud y\right) \ud x
 = - \left[ \frac{1}{2} \left( \int_0^x \eta \, p \ud y\right)^2 \right]_0^1= 0\,, 
\end{equation*}
since $\int_0^1 q(t,x) \ud x = 0$ for all $t\in [0,1]$. 
\par For the scalar product with $1$ of the r.h.s. term in the momentum equation of \eqref{GeodesicEquationL2}, we use a change of variable by $\varphi$ to obtain
\begin{multline*}
\left\langle 1,\frac 12 \left(\int_x^1 \eta(q)p^2 \ud y \right)\right\rangle_\eta = \int_0^1 \frac 12 \left(\int_x^1 \eta(q)p^2 \ud y \right)\,\eta \ud x \\= \frac 12 \int_0^1 \int_x^1 p^2 \circ \varphi^{-1} \ud y \ud x = \frac 12 \int_0^1 x \, p^2 \circ \varphi^{-1} \ud x = a\,.
\end{multline*}
Therefore we get the result in the formula \eqref{eq:FirstProjection}.
\par
After some computations, we also have $\frac {\ud}{\ud t}\int_0^1 \varphi \,\eta \,p \ud x = 0$ which implies
\begin{align*}
\langle \,\dot{p} \,,\, \varphi \, \rangle_\eta   &= - \int_0^1 p\, \varphi \,\eta \left(\int_0^x \eta \,p \ud y \right)\ud x - \int_0^1 p \,\varphi \,\eta \left(\int_0^x \eta \left( \int_0^y p\, \eta \ud z\right) \ud y \right)\ud x \\
 &=\frac 12 \int_0^1 \eta  \left( \int_0^x \eta \,p \ud y \right)^2 \ud x + \int_0^1 \eta \left( \int_0^x \eta \, p \ud y \right)^2 \ud x\\
 & = \frac 32 \int_0^1 \left( \int_0^x p \circ \varphi^{-1}  \right)^2\ud x = b\,, 
\end{align*}
where the second line is obtained by integration by part and the last one by the change of variable with $\varphi$.
Last, we compute, by integration by part
\begin{equation*}
\left\langle \varphi,\frac 12 \left(\int_x^1 \eta(q)p^2 \ud y \right)\right\rangle_\eta = \frac12 \int_0^1 \eta \, \varphi \left(\int_x^1 \eta p^2 \ud y \right)\ud x =  \frac 14 \int_0^1 \varphi^2 \, \eta p^2\ud x = \frac 14 \int_0^1 x^2 \, p^2 \circ \varphi^{-1} \ud x =c\,.
\end{equation*}
%
%
%
%
The geodesic equation then reads:
\begin{equation}
\begin{cases}
\dot{q} = \eta(q)p\\
\dot{p} = - \frac 12 \int_x^1 \eta(q)p^2 \ud y +   \begin{bmatrix} 1 & \varphi \end{bmatrix} H_2^{-1}  \begin{bmatrix} a \\ b +c \end{bmatrix} 
\end{cases}
\end{equation}
where the coefficients $a,b,c$ are defined in the Theorem.
\end{proof}

\section{Other proofs}\label{Ap:Proofs}
\begin{proof}[Proof of Proposition \ref{Th:ExampleOfGeodesic}]
For instance, consider $n=2$, $\lambda$ a positive real and $\alpha \in ]0,\frac 12[$, $q_1(0) = 1/2- \alpha$, $p_1(0) = \lambda$ and $q_2(0) = 1/2+\alpha$, $p_2(0) = -\lambda$, then the solution has the following properties.

First, by (central) symmetry with respect to $1/2$, $v(t,1/2)=0$ for all time $t \in \R_+$ and as a consequence, $1/2$ is a fixed point of the flow, that is $\varphi(t,\frac 12) = \frac 12$.
Second, $\lim_{t \to \infty} q_i(t) = 1/2$ for $i=1,2$ indeed, it is easy to see that the trajectory of $q_i$ is monotonic. Then, since the manifold of landmarks is complete, $(q_1,q_2)$ has to escape every compact set which implies the desired property.
The jacobian of the flow satisfies
\begin{equation}\label{Eq:Jacobian}
\partial_t \varphi_x(t,1/2) = \partial_x v(t,1/2) \varphi_x(t,1/2)\,.
\end{equation}
Therefore, since $\partial_x v(t,1/2)$ is always negative, it implies that $\varphi_x(t,1/2)$ is decreasing and remains positive. Then, using the fact that $\lim_{t \to \infty} q_i(t) = 1/2$ for $i=1,2$, one can conclude that 
$\lim_{t \to \infty} \varphi_x(t,1/2) = 0$ which is the last property we will need. We summarize what we have shown as follows:

Of course, the argument developed above would be valid for a larger class of initial momentums with a central symmetry but this is not needed.
\end{proof}

\begin{proof}[Proof of Theorem \ref{th:Existence}]
Local existence follows from existence of Caratheodory ordinary differential equations \cite{ODE}. Simple estimates give that the right-hand side is  Lipschitz continuous w.r.t. $q \in L^\infty([0,1])$ on every bounded balls: 
Indeed, on the ball $B(0,r)$ in $L^\infty([0,1])$, the map $q \mapsto \exp(\int_0^x q(y) \ud y)$ is Lipschitz with the Lipschitz constant corresponding to that of $\exp$ on $[-r,r]$ which we denote by $K$. More precisely, for $q_1,q_2 \in L^\infty([0,1])$  we have
\begin{equation}
 \eta(q_1) \pi^*_{q_1}(p) - \eta(q_2) \pi^*_{q_2}(p) = \left( \pi^*_{q_1}(p) \right)\left( \eta(q_1) - \eta(q_2) \right) + \eta(q_2)\left( \pi^*_{q_1}(p) -  \pi^*_{q_2}(p) \right)\,.
\end{equation}
Using inequalities \eqref{Eq:1}, \eqref{Eq:2} and \eqref{Eq:3} below, we obtain
\begin{align*}
 \| \eta(q_1) \pi^*_{q_1}(p) - \eta(q_2) \pi^*_{q_2}(p) \|_\infty &\leq K \| \pi^*_{q_1}(p) \|_\infty \| q_1 - q_2 \|_\infty + e^r \| \pi^*_{q_2}(p) - \pi^*_{q_1}(p) \|_\infty \\
 & \leq 3K \| p \|_\infty \| q_1 - q_2 \|_\infty + 2 e^r  \|  p \|_{\infty} K \| q_1- q_2 \|_\infty \\
 & \leq (3 + 2e^r)K \| p \|_\infty \| q_1 - q_2 \|_\infty\,.
\end{align*}
As for the time regularity of $\|p(t)\|_\infty$, the hypothesis $L^2$ is sufficient as proven in \cite{ODE}. 
\par 
In order to prove global existence, we provide some a priori estimates on the norm of $q$ that bound its growth. For that, we first remark that for any $q \in L^2([0,1])$, we have 
\begin{subequations}
\begin{align}
&\| \varphi \|_{\infty} \leq 1 \, \label{Eq:1}\\ 
&| \langle p,\eta(q) \rangle | \leq \| p \|_{\infty} \| \eta(q) \|_{L^1} = \| p \|_{\infty} \,,\label{Eq:2}\\
&| \langle p,\varphi \, \eta(q) \rangle | \leq \| p\, \varphi \|_{\infty} \| \eta(q) \|_{L^1} \leq \| p \|_{\infty}\label{Eq:3}\,. 
\end{align}
\end{subequations}
These inequalities imply first that 
 $\| \pi^*_{q}(p) \|_\infty$ is bounded and also that $\dot{q}$ is bounded in $L^1$.
 Indeed, one has, using $\| \eta(q) \|_{L^1} = 1$,
 \begin{align*} 
 &\| \pi^*_{q}(p) \|_\infty \leq 3\| p \|_{\infty}\\
 &\|\eta(q) \pi^*_q(p)\|_{L^1} \leq \| \eta(q) \|_{L^1} \| \pi^*_{q}(p) \|_\infty \leq  3\| p \|_{\infty} \,.
\end{align*}
Now, since $q$ is bounded in $L^1$ (using initial condition), we get 
\begin{equation} \label{Eq:Boundedness}
\| \eta(q) \|_{\infty} \leq e^{\|q(0)\|_{\infty} + \int_0^t 3\| p(s) \|_{\infty}\ud s}
\end{equation} and therefore
\begin{equation}\label{Eq:Boundedness2}
\| q(t) \|_{L^\infty} \leq  \| q(0) \|_{L^\infty} + K \int_0^t \| p(s) \|_{\infty}(e^{\|q(0)\|_{\infty} + \int_0^s 3\| p(u) \|_{\infty}\ud u}) \, \ud s\,,
\end{equation}
which implies global existence.
\end{proof}

\begin{proof}[Proof of Lemma \ref{th:WeakConvergenceLemma}]
Note that the second point is a direct consequence of the first one since $\varphi(x) = \int_0^x \eta \ud y$, therefore the convergence of $\eta$ implies the convergence of the first spatial derivative of $\varphi$.
The first point follows by application of the Aubin-Lions-Simon lemma \cite{aubin,lions,simon} on $\int_0^x q \ud y \in H^1([0,1],H^1([0,1])$ which is compactly embedded in $C^0([0,1],C^0([0,1]))$. As a consequence, it is not difficult to prove that it is also compactly embedded in $C^0(D)$.

We now prove the last two points. We refer to formulas \eqref{eq:projection} and \eqref{eq:projectionDuale} which involve $\eta, \varphi$. Let us recall for readability the projection $\pi_q$:
\begin{equation}\label{eq:DetailedProjection}
\pi_q(f) = f -  \begin{bmatrix} \eta & \varphi \eta \end{bmatrix} H_2^{-1}  \begin{bmatrix} \langle f , \eta \rangle_{1/\eta}  \\ \langle f , \varphi \,\eta \rangle_{1/\eta}  \end{bmatrix}\,.
\end{equation}

The two functions $\eta_n$ and $\varphi_n$ are strongly convergent in $C^0(D)$ (and therefore in $L^2$). Thus, it is also the case for polynomial functions of $\eta_n $ and $\varphi_n$ since $C^0(D)$ is a Banach algebra. Since the projections $\pi_{q_n}$ and $\pi_{q_n}^*$ only involve strongly convergent functions and associated scalar product, it gives the result.
The last point follows the same line since one has
\begin{equation}\label{eq:DerivativeProjection}
\frac{\ud}{\ud t} \pi_{q}(f) = -  \begin{bmatrix} \dot{\eta} & \dot{\varphi} \eta + \varphi \dot{\eta} \end{bmatrix} H_2^{-1}  \begin{bmatrix} \langle f , \eta \rangle_{1/\eta}  \\ \langle f , \varphi \eta \rangle_{1/\eta}  \end{bmatrix}  - \begin{bmatrix} \eta & \varphi \eta \end{bmatrix} H_2^{-1}  \begin{bmatrix} 0  \\ \langle f , \dot{\varphi} \eta \rangle_{1/\eta}  \end{bmatrix}\,.
\end{equation}
We check that each term in the formula \eqref{eq:DerivativeProjection} are strongly convergent using 
\begin{align*}
& \dot{\eta} = \eta \int_0^x \dot{q} \ud y \\
& \dot{\varphi} = \int_0^z \eta \left( \int_0^x \dot{q} \ud y \right)\ud x \,.
\end{align*}

Using the assumption $p \in H^1([0,1],L^2([0,1]))$, we have $\dot{q} = \eta(q)\, \pi_q^*(p)$ and also $\int_0^x \dot{q}_n \ud y \in H^1([0,1],H^1([0,1]))$. Using the compact embedding in $C^0(D)$, it implies the strong convergence of $\int_0^x \dot{q} \ud y$ in $C^0(D)$. Then, the result follows easily for $\frac{\ud}{\ud t}  \pi_q$ and it is similar for $\frac{\ud}{\ud t}  \pi_q^*$.
\par
The last point follows from the fact that the projection part in \eqref{eq:DetailedProjection} involves dual pairings between vectors $\varphi_n,\eta_n$ that are strongly convergent in $C^0(D)$ (and therefore strong convergence in $L^2$) $z_n$ which weakly converge. The same argument also applies to $\left(\frac d {dt} \pi_{q_n}\right)(z_n)$.
\end{proof}

\begin{proof}[Proof of Proposition \ref{th:Boundedness}]
From the proof of Proposition \ref{th:Existence}, formulas \eqref{Eq:Boundedness} and \eqref{Eq:Boundedness2} give that  $q_n$ is bounded in $H^1([0,1],L^\infty([0,1]))$ and therefore bounded in $H^1([0,1],L^2([0,1]))$. Proving the weak convergence can thus be done on every converging subsequence. 
We consider now a weakly converging subsequence also denoted by $q_n$ to $\tilde{q} \in H^1([0,1],L^2([0,1]))$, then Lemma \ref{th:WeakConvergenceLemma} implies the strong convergence of $\pi_{q_n}^*$ and $\eta(q_n)$ to $\pi_{\tilde{q}}^*$ and $\eta(\tilde{q})$. As a consequence, $\eta(q_n)\pi_{q_n}^*(p_n)$ weakly converges in $L^2(D)$ to $\eta(\tilde{q})\pi_{\tilde{q}}^*(p_\infty)$ we get that $\tilde{q}$ is a solution associated with $p_\infty$ and that this solution is also in $H^1([0,1],L^\infty([0,1]))$. Therefore, $\tilde{q}$ is the unique solution given by Proposition \ref{th:Existence}, that is $\tilde{q} = q_\infty$. It implies that all the converging subsequences of $q_n \in H^1([0,1],L^2([0,1]))$ converges to $q_\infty$ which gives the result.
\end{proof}

\begin{proof}[Proof of Lemma \ref{th:L1Approximation}]
We look for $p_n(t,x)$ under the following form
\begin{equation*}
p_n(t,x) = (a(t,x) + b(t,x) \cos(2\pi n x)) \sin(2\pi n x)\,.
\end{equation*}
It is clear that the first condition $p_n\rightharpoonup 0$ is satisfied using the identity $$ 2\cos(2\pi n x) \sin(2\pi n x)= \sin(4\pi n x)\,.$$
Now we expand $p_n^2$ to obtain
\begin{multline*}
p_n^2(t,x) = a^2(t,x)\sin^2(2\pi n x) + b(t,x)^2 \cos^2(2\pi n x)\sin^2(2\pi n x) \\+ 2  a(t,x)b(t,x) \sin^2(2\pi n x) \cos(2\pi n x)\,.
\end{multline*}
We observe that 
\begin{equation*}
\sin^2(2\pi n x) \cos(2\pi n x) = \frac 12 \cos(2\pi n x) - \frac 14 (\cos(6\pi n x) + \cos(2\pi n x))
\end{equation*}
which weakly converges to $0$.
Therefore, we get that 
\begin{equation*}
p_n^2 \rightharpoonup \frac 12 a^2 + \frac 14 b^2\,,
\end{equation*}
since $\sin^2(2\pi n x),\cos^2(2\pi n x) \rightharpoonup \frac 12 $.
Similarly, we have
\begin{multline*}
\dot{p}_n^2(t,x) = \dot{a}^2(t,x)\sin^2(2\pi n x) + \dot{b}(t,x)^2 \cos^2(2\pi n x)\sin^2(2\pi n x) \\+ 2  \dot{a}(t,x)\dot{b}(t,x) \sin^2(2\pi n x) \cos(2\pi n x)\,,
\end{multline*}
and 
\begin{equation*}
\dot{p}_n^2 \rightharpoonup \frac 12 \dot{a}^2 + \frac 14 \dot{b}^2\,.
\end{equation*}
We want to find $a,b$ such that 
\begin{align*}
& \frac 12 a^2 + \frac 14 b^2 = \mu \\
 &\frac 12 \dot{a}^2 + \frac 14 \dot{b}^2 = \nu\,.
\end{align*}
with initial conditions $a(0)=0$ and $b(0)=0$.

We solve this equation in polar coordinates and set $(a/\sqrt{2},b/2) = r \exp(i\theta)$ to obtain
\begin{equation}\label{eq:SystemPolar}
\begin{cases}
\frac 12 a^2 + \frac 14 b^2 = r^2\\
\frac 12 \dot{a}^2 + \frac 14 \dot{b}^2 = \dot{r}^2 + r^2 \dot{\theta}^2
\end{cases}
\end{equation}
This implies that $r = \sqrt{\mu} \in L^2(D)$ and therefore $\dot{r}^2 = (\partial_t \sqrt{\mu})^2$.
A necessary and sufficient condition for the existence of a solution to the system \eqref{eq:SystemPolar} is
\begin{equation}\label{eq:NecessaryCondition}
(\partial_t \sqrt{\mu})^2 \leq \nu \,.
\end{equation}
which is exactly another way to write condition \eqref{eq:InequalityCondition}. Indeed, one has
\begin{equation*}
\dot{\theta}^2 = \frac{\nu - (\partial_t \sqrt{\mu})^2}{\mu} \in L^1(D,\mu)\,,
\end{equation*}
for which the following formula is a particular solution
\begin{equation*}
\theta(t) =\theta(t=0) + \int_0^t \sqrt{\frac{\nu - (\partial_t \sqrt{\mu})^2}{\mu}} \, \ud s    \,,
\end{equation*}
and this implies that $\theta \in L^2(D,\mu)$. Now, this in turn implies that $r\sin(\theta)$ and $r\cos(\theta)$ belong to $H^1([0,1],L^2([0,1]))$.

The initial condition $p(0)=0$ is implied by the fact that $|r\sin(\theta)|(0)\leq \sqrt{\mu}(0) = 0$ and  $|r\cos(\theta)|(0)\leq \sqrt{\mu}(0) = 0$.

The last assertion of the lemma is obvious. Note that if $\mu(1)=0$ then $a(1) = b(1) = 0$, this case will also be used in the relaxation with boundary constraints.
\end{proof}

%

\begin{proof}[Proof of Lemma \ref{th:FirstStepDensity}]
We use Lemma \ref{th:L1Approximation} to obtain $\alpha_n \in L^2(D)$ such that $\alpha_n\rightharpoonup 0$, $\alpha_n^2 \rightharpoonup \mu - p^2$ and $\dot{\alpha}_n^2 \rightharpoonup \nu - \dot{p}^2$. Since $(p,q) \in H^1([0,1],L^\infty([0,1])) \times H^2([0,1],L^\infty([0,1]))$ and $(p,q,\mu,\nu) \in \mathcal{C}_R \cap \mathcal{C} \times L^\infty(D) \times L^1(D)$ we have that $p_n \eqdef p + \alpha_n \in L^\infty(D)$ is bounded uniformly in $n$.
Therefore, Lemma \ref{th:Existence} provides a solution to 
\begin{equation*}
\dot{q}_n = \eta(q_n) \pi^*_{q_n}(p_n)\,.
\end{equation*}
and this solution $q_n$ weakly converges to $q$ in $L^2(D)$ which is the solution associated with $p$. 

Now, we show that $(\pi^*_{q_n}(p_n),q_n)$ satisfies the requirements. Indeed, $\pi^*_{q_n}(p_n) = \pi^*_{q_n}(p) + \pi^*_{q_n}(\alpha_n)$. The first term strongly converges to $\pi^*_q(p)$. Moreover, the second term satisfies the identity
$$ \pi^*_{q_n}(\alpha_n) = \alpha_n + \beta_n$$ with $\beta_n$ that strongly converges to $0$ ($\beta_n$ is the projection of $\alpha_n$ on a finite dimensional subspace). We then write $\pi^*_{q_n}(p_n) = \pi^*_{q_n}(p) + \beta_n + \alpha_n$ with $\pi^*_{q_n}(p) + \beta_n$ strongly converging to $\pi^*_q(p)$ and we apply the remark \ref{rem:SimpleLemma} below to obtain $\pi^*_{q_n}(p_n)^2 \rightarrow \mu$.

We now prove in the same way the result for $\frac{\ud }{\ud t} \left(\pi^*_{q_n}(p_n) \right)$ using
\begin{equation*}
\frac{\ud }{\ud t} \pi^*_{q_n}(p_n) = \dot{p}_n + \left(\frac{\ud }{\ud t} \pi^*_{q_n}\right)(p_n)\,.
\end{equation*}
Lemma \ref{th:WeakConvergenceLemma} applies and the second term strongly converges to $0$. Then, Remark \ref{rem:SimpleLemma} below gives that $\left(\frac{\ud }{\ud t} \left(\pi^*_{q_n}(p_n) \right)\right)^2 \rightarrow \nu$.
Redefining $p_n$ by $\pi^*_{q_n}(p_n) $ gives the result.
\end{proof}

\begin{remark}\label{rem:SimpleLemma}
Let $X \subset \R^d$ be a compact domain and $\mu$ be a Borel measure on 	$X$. In $L^2(X,\mu)$, we consider $f_n \rightarrow f$ a strongly convergent sequence and $g_n \rightharpoonup 0$ a weakly convergent sequence. We also assume the weak convergence $g_n^2 \rightarrow z$ in $\mathcal{M}_+(X)$.
Then one has $$ (f_n + g_n)^2 \rightarrow  x^2 + z \text{ in }\mathcal{M}_+(X)\,.$$
Indeed, one has $(f_n + g_n)^2 = f_n^2 + g_n^2 + 2f_ng_n$. Moreover, for every $\varphi \in C^0(X)$ 
$$\int f_ng_n \varphi \ud \mu \rightarrow 0$$
since $\varphi f_n $ strongly converges to $\varphi f$, which proves the result.
\end{remark}

\begin{proof}[Proof of Lemma \ref{Th:LemmaSimpleConvexity}]
Let $x$ be a minimum of $f+g \circ A$. Note that it is sufficient to prove that $Ax_0=Ax$ since in this case, $g(Ax_0) = g(Ax)$ and therefore, since the value of the function $f+g\circ A$ is the same at $x_0$ and $x$, we have that $f(x) = f(x_0)$. This implies that $x_0=x$ by the hypothesis on $f$.

\par
We now prove $Ax_0=Ax$. If it were not the case, then by strict convexity of $g$ on the segment $[Ax_0,Ax]$ and convexity of $f$, we have that 
\begin{align*}
g( 1/2 (Ax_0+Ax)) <  1/2 g(Ax_0) +  1/2 g(Ax) \\
f( 1/2 (x_0+x)) \leq 1/2 f(x_0) +  1/2 f(x) \,.
\end{align*}
which implies 
$$g( 1/2 (Ax_0+Ax))+f( 1/2 (Ax_0+Ax))<f(x_0)+g(Ax_0)\,,$$
which gives a contradiction.
\end{proof}

\begin{proof}[Proof of Lemma \ref{Th:GeneralLemma}]
If $v \in \partial f(x)$, then for every $z \in X$, 
\begin{equation} \label{Eq:DefSousDiff}
f(z) \geq f(x) +  \langle z-x , v \rangle\,.
\end{equation}
For $z= \lambda y$ and $\lambda \to +\infty$, the previous inequality implies 
$f(y) \geq \langle y , v \rangle$ which is the first condition.
Taking $\lambda = 0$ implies $\langle x , v \rangle \geq f(x)$, which gives, together with the previous inequality, the second condition.
\par
If the two conditions are fulfilled then, adding the equality of the second condition to the inequality of the first one gives inequality \eqref{Eq:DefSousDiff}.
\end{proof}

\section{Reproducing kernel Hilbert space of $H_0^2([0,1])$} \label{Ap:rkhs}

The reproducing kernel for the space $H^2([0,1])$ endowed with the metric 
\begin{equation}\label{Eq:NormRK1}\| f  \|^2 = f(0)^2 + f'(0)^2 + \int_0^1 f''(s)^2 \ud s \,\end{equation} is known to be \cite[Section 1.6.2]{berlinet},
\begin{equation}
K(s,t) = 1 + st + \int_0^1 (t-u)_+(s-u)_+ \ud u
 \end{equation} 
 and equivalently,
 \begin{equation}
K(s,t) = 
\begin{cases}
1 + st + \frac{1}{2} ts^2-\frac 16 s^3 \text{ if } s<t \\
1 + st + \frac{1}{2} st^2-\frac 16 t^3 \text{ otherwise}\,.
\end{cases}
\end{equation}
Now, we want to use this formula to give an explicit expression of the kernel of $H_0^2([0,1])$ endowed with the norm 
$\int_0^1 f''(s)^2 \ud s$ which coincides with the norm \eqref{Eq:NormRK1}.
Therefore, this situation is a particular case of computing the kernel associated with a closed subspace $G$ of an initial reproducing Hilbert space $H$. In such a case,  the orthogonal sum $G \oplus G^\perp$ implies that $K_H = K_{G} + K_{G^\perp}$ (see \cite[Section 1.3]{berlinet}). Let $p$ be the orthogonal projection on $G$, then the reproducing kernel $K_G$ can be expressed as follows
\begin{equation}
K_{G}(s,t) = \langle p(K(s,\cdot)), p(K(t,\cdot))\rangle \,. 
\end{equation}
In our setting, we have $G=H_0^2([0,1])$. For an explicit expression of the kernel, one has, following for instance \cite[Section 1.2]{berlinet},
\begin{equation}
K_{G}(s,t) = \frac{\begin{vmatrix}
  K(s,t) & K(1,t) & \partial_1K(1,t) \\
  K(s,1) & K(1,1) & \partial_1 K(1,1) \\
  \partial_2 K(s,1)& \partial_2 K(1,1) & \partial_{1,2} K(1,1)
 \end{vmatrix}}{\begin{vmatrix}
  K(1,1) & \partial_1 K(1,1) \\
  \partial_2 K(1,1) & \partial_{1,2} K(1,1)
 \end{vmatrix}}\,.
\end{equation}
More explicitely, this gives
\begin{equation}
K_{G}(s,t) = K(s,t) +\left(-1 + \frac{1}{2}(s+t) - \frac{1}{3}ts\right)(ts)^2 \,.
\end{equation}

\bibliographystyle{plain}      
\bibliography{articles,SecOrdLandBig}   

\end{document}